\documentclass[12pt,reqno]{amsart}
\pdfoutput=1
\usepackage[headings]{fullpage}
\usepackage{amssymb,amsmath,bbm,tikz}
\usepackage{graphicx,color,enumerate, float}
\usepackage[all,cmtip]{xy}
\usepackage{url}
\usepackage[margin=1.2in]{geometry}
\usepackage{verbatim}
\usepackage{pb-diagram}

\usepackage[bookmarks=true,%
    colorlinks=true,%
    linkcolor=blue,%
    citecolor=blue,%
    filecolor=blue,%
    menucolor=blue,%
    urlcolor=blue,%
    breaklinks=true]{hyperref}

\newtheorem{theorem}{Theorem}[section]
\theoremstyle{definition}
\newtheorem{proposition}[theorem]{Proposition}
\newtheorem{lemma}[theorem]{Lemma}
\newtheorem{definition}[theorem]{Definition}
\newtheorem{remark}[theorem]{Remark}
\newtheorem{corollary}[theorem]{Corollary}
\newtheorem{conjecture}[theorem]{Conjecture}

\newtheorem{example}[theorem]{Example}

\renewcommand{\setminus}{{\smallsetminus}}

\def\be{\begin{equation}}
\def\ee{\end{equation}}

\graphicspath{{figures/}}

\begin{document}

\title{Schmutz-Thurston Duality} 

\author{Ingrid Irmer}
\address{SUSTech International Center for Mathematics\\
Southern University of Science and Technology\\Shenzhen, China
}
\address{Department of Mathematics\\
Southern University of Science and Technology\\Shenzhen, China
}
\email{ingridmary@sustech.edu.cn}


\begin{abstract} Two people who pioneered the study of mapping class group-equivariant deformation retractions of Teichm\"uller space of closed compact surfaces are Schmutz Schaller and Thurston. This paper studies how the two different approaches are dual to each other.  
\end{abstract}

\maketitle

{\footnotesize
\tableofcontents
}


\section{Introduction}
\label{sec.intro}

The closed, connected topological surface of genus $g\geq 2$ will be denoted by $\mathcal{S}_{g}$, the Teichm\"uller space of $\mathcal{S}_{g}$ by $\mathcal{T}_{g}$ and the mapping class group of $\mathcal{S}_{g}$ by $\Gamma_{g}$. Thurston defined a $\Gamma_{g}$-equivariant deformation retraction of $\mathcal{T}_{g}$ onto the Thurston spine, $\mathcal{P}_{g}$, \cite{Thurston}, \cite{ThurstonSpineSurvey}, \cite{MS}. The latter consists of all the points in $\mathcal{T}_{g}$ at which the set of shortest geodesics, the systoles, cuts the surface into polygons. \\

Studying the Thurston spine is closely tied up with studying the systole function $f_{\mathrm{sys}}$ on $\mathcal{T}_{g}$. The systole function is the map from $\mathcal{T}_{g}$ to $\mathbb{R}_{+}$ whose value at any point is given by the length of the systoles of the corresponding marked hyperbolic surface. The systole function is known to be a topological Morse function \cite{Akrout}, \cite{SchmutzMorse}. Topological Morse functions have nondegenerate, isolated critical points and stable/unstable manifolds of critical points, \cite{Morse}. However, unlike the smooth case, even in the presence of a metric, the stable/unstable manifolds are not uniquely defined. One motivation for the results in this paper is to replace the stable/unstable manifolds with canonical objects, namely the Thurston spine and related sets of minima. The critical points of $f_{\mathrm{sys}}$ are all contained in $\mathcal{P}_{g}$, \cite{Thurston}. Technical terms, such as topological Morse function, will be defined in Section \ref{defns}.\\

The intuitive idea is that the Thurston spine contains a choice of unstable manifolds of the critical points of the systole function, whereas a choice of stable manifolds are given by corresponding sets of minima defined by Schmutz Schaller, \cite{SchmutzMorse}. The first part of this statement was proven in \cite{MS}, and one of the goals of this paper is to justify the second part of the statement. One of the ingredients, expressed in the next theorem, is that if $p$ is a local maximum of $f_{\mathrm{sys}}$ with set of systoles $C$, the ``faces'' of the set of minima $\mathrm{Min}(C)$ are labelled by sets of curves corresponding to the sets of curves realised as systoles on the strata adjacent to $p$. Polytopal regularity and equivalence classes are defined in Subsection \ref{backgroundforthm}.

\begin{theorem}
\label{onlyfaces}
Suppose $C$ is a set of curves for which $\mathrm{Min}(C)$ has polytopal regularity and $p$ is a local maximum of $f_{\mathrm{sys}}$ with set of systoles $C$. The sets of minima making up the faces of $\mathrm{Min}(C)$ are in 1-1 correspondence with equivalence classes of strata adjacent to $p$ in $\mathcal{P}_{g}$.
\end{theorem}

It was shown in \cite{Akrout} that a critical point $p$ of $f_{\mathrm{sys}}$ is a point at which, for some set $C$,
\begin{itemize}
\item{$C$ is the set of systoles at $p$, and}
\item{$p$ is in $\mathrm{Min}(C)$.}
\end{itemize}
The purpose of the next theorem is to show that, for a critical point in a locally top-dimensional cell of $\mathcal{P}_{g}$, $\mathrm{Min}(C)$ has the right intersection properties with $\mathcal{P}_{g}$ to be considered a dual. The difficulty here is that, geometrically, $\mathcal{P}_{g}$ as well as the sets of minima can be quite pathological objects; these are not the unual intersections of embedded submanifolds. By construction, the pre-image of a point of $\mathcal{P}_{g}$ under Thurston's deformation retraction intersects $\mathcal{P}_{g}$ in a single point. Since Thurston's deformation retraction is constructed from an $f_{\mathrm{sys}}$-increasing flow, such a pre-image is a valid choice of stable manifold.

\begin{theorem}
\label{neededforduality}
Suppose $p$ is a critical point of $f_{\mathrm{sys}}$ with set of systoles $C$ and $T_{p}\mathrm{Min}(C)$ does not have any vectors in the tangent cone to $\mathcal{P}_{g}$. Then $\mathrm{Min}(C)$ is a cell with empty boundary. In addition, there is a homotopy of $\mathrm{Min}(C)$ to a cell obtained as the pre-image of $p$ under Thurston's equivariant deformation retraction of $\mathcal{T}_{g}$ onto $\mathcal{P}_{g}$. This homotopy fixes $p$ and keeps the points in the thin part of $\mathrm{Min}(C)$ in the thin part of $\mathcal{T}_{g}$.
\end{theorem}

At a critical point described by Theorem \ref{neededforduality}, the assumptions of the theorem together with Proposition 1.3 of \cite{MS} imply that $\mathcal{P}_{g}$ has a tangent space at $p$ of dimension equal to the codimension of $\mathrm{Min}(C)$. \\

In topology, stable and unstable manifolds of critical points are standard tools for constructing cell decompositions of manifolds. In the case of Teichm\"uller space of punctured surfaces, $\Gamma_{g}$-equivariant cell decompositions of a bordification of Teichm\"uller space were obtained in \cite{BE88}, \cite{Harer} and \cite{PennerComplex} or by using topological Morse functions in \cite{SchmutzCelldecomp2}. Sets of minima have many properties analogous to the cells of \cite{BE88}, \cite{Harer} and \cite{PennerComplex}. As discussed briefly at the end of Section \ref{constructingduals}, the closest puncture-free, marked point-free analogue of the cell decompositions of $\mathcal{T}_{g}$ appears to be a ``pinched cell decomposition'' using sets of minima.\\

In the process of proving Theorems \ref{onlyfaces} and \ref{neededforduality}, the sets of curves that become arbitrarily short on a set of minima were completely determined. This gives a puncture-free analogue of the screens of Penner-McShane \cite{PM} and of a construction due to Harer \cite{Harer}. As explained in Section \ref{horizonsec}, the thin parts of sets of minima have a combinatorial structure that can be identified with a subcomplex of the barycentric subdivision of Harvey's curve complex. By analogy with punctured surfaces, it would seem that the horizon map uniquely determines a set of minima. If so, this implies the index of a critical point can be computed combinatorially under fairly general assumptions. Previously, except when the set of systoles at the critical point does not contain any proper filling subsets (see Section 3 of \cite{MS}) this could only be done numerically.\\

\textbf{Outline of the paper.} This paper is a sequel to \cite{MS}, and relies on a a number of results about sets of minima, loci and the Thurston spine. These results are surveyed in Section \ref{defns} in an attempt at self-containment. Section \ref{horizonsec} defines the horizon map and explains how to calculate it. The horizon map determines the asymptotic properties of the combinatorial description of sets of minima discussed in Section \ref{twocases}. Theorem \ref{onlyfaces} is also proven in this Section. Finally, in Section \ref{dualitysec}, Theorem \ref{neededforduality} is proven as part of a discussion of what a dual to the Thurston spine should be, and how to construct duals from sets of minima. 



\section{Notation and background}
\label{defns}

The purpose of this section is to provide the important definitions and background that will be used throughout this paper. \\

$\mathcal{S}_{g}$ denotes a closed, orientable topological surface of genus $g\geq 2$, and $S_{g}$ will be used for the surface $\mathcal{S}_{g}$ endowed with a marked hyperbolic structure corresponding to a point in $\mathcal{T}_{g}$. The symbol $\Gamma_{g}$ will be used for the mapping class group of $\mathcal{S}_{g}$. Curves are simple, closed, nontrivial isotopy classes on $\mathcal{S}_{g}$. On $S_{g}$, a curve will also be thought of as the unique geodesic representative of its isotopy class. A curve will usually be denoted by a lower case $c$, whereas uppercase $C$ is used for a finite set of curves, with pairwise geometric intersection number at most one. The geometric intersection number of two curves is the minimum number of crossings, where this minimum is evaluated over pairs of representatives of the isotopy classes that intersect transversely.\\

A point in $\mathcal{T}_{g}$ representing a marked hyperbolic surface will usually be denoted by lowercase $x$, unless it is required to satisfy some properties that make it rigid, such as being a critical point of $f_{\mathrm{sys}}$, which will then be denoted by $p$.\\

A set of curves $C$ in $\mathcal{S}_{g}$ will be said to \textit{fill} the surface if the complement on $\mathcal{S}_{g}$ of a set of representatives in minimal position consists of a union of topological disks. Equivalently, $C$ fills $\mathcal{S}_{g}$ iff every curve on $\mathcal{S}_{g}$ has nonzero geometric intersection number with at least one curve in $C$.\\

The length of a curve $c$ is an analytic function $L(c)$ from $\mathcal{T}_{g}$ to $\mathbb{R}_{+}$ given by the length of the geodesic representative. The length of a curve is a special case of a so-called length function.

\begin{definition}[Length function]
A finite ordered set of curves $C=(c_{1}, \ldots, c_{n})$ together with a finite ordered set of real, positive weights $A=(a_{1}, \ldots, a_{k})$ define an analytic function $L(A,C):\mathcal{T}_{g}\rightarrow \mathbb{R}^{+}$ , called a length function, given by
\begin{equation*}
L(A,C)(x)\,=\, \sum_{j=1}^{k} a_{j}L(c_{j})(x)
\end{equation*}

\end{definition}
Length functions satisfy many convexity properties, \cite{Bestvina}, \cite{Kerckhoff} and \cite{Wolpert}. For example, they are strictly convex with respect to the Weil-Petersson metric.\\

The interior of a $k$-cell is homeomorphic to an open ball in $\mathbb{R}^{k}$. In the literature, usually a cell is assumed to be closed. However, when working with sets of minima, some of the boundary points of the open ball in $\mathbb{R}^{k}$ might be mapped to points in $\mathcal{T}_{g}$, while some might only be mapped to points in a bordification of $\mathcal{T}_{g}$. As $\mathcal{T}_{g}$ is open, cell decompositions are usualy only of some choice of bordification of $\mathcal{T}_{g}$. For this reason, cells in this paper may be open, closed or neither, depending on the context. In addition, the interior of all cells will be continuously differentiable. When studying homotopies of cells, for simplicity the same symbol will be used for the cell and the map of the ball into $\mathcal{T}_{g}$. Vertices, edges and faces of a cell will automatically refer to the vertices, edges and faces of the closure of the cell in $\mathcal{T}_{g}$.\\

The statement in Theorem \ref{neededforduality} that a cell has empty boundary means that any boundary points of the cell should be thought of as noded surfaces in the metric completion of $\mathcal{T}_{g}$ with respect to the Weil-Petersson metric.\\

$\mathrm{Sys}(C)$ will be used to denote set of points in $\mathcal{T}_{g}$ on which $C$ is the set of equal shortest curves (the set of systoles). Following Thurston, $\mathrm{Sys}(C)$ will also be called a stratum. This stratification of $\mathcal{T}_{g}$ extends to the completion of $\mathcal{T}_{g}$ with respect to the Weil-Petersson metric. The symbol $\mathrm{Sys}(m)^{\infty}$ will be used to denote the set of points in the Weil-Petersson completion of $\mathcal{T}_{g}$ corresponding to noded surfaces along which the multicurve $m$ has been pinched.\\

 The Thurston spine, $\mathcal{P}_{g}$, is the union of strata labelled by filling sets of curves. It follows from a theorem of Lojasiewicz, \cite{Lojasiewicz1964}, that there is a triangulation of $\mathcal{P}_{g}$ compatible with the stratification, i.e. it is possible to triangulate $\mathcal{P}_{g}$ in such a way that the interior of each simplex is contained in a single straum. As a cell complex in $\mathcal{T}_{g}$, it is not the case that the dimension of $\mathcal{P}_{g}$ is locally constant. A locally top-dimensional cell of $\mathcal{P}_{g}$ is therefore defined to be a cell, every point of which has a neighbourhod in $\mathcal{T}_{g}$ that does not intersect a larger dimensional cell of $\mathcal{P}_{g}$.\\

A set of systoles has the property that any two systoles in the set intersect in at most one point. This can be proven by contradiction; assume more than two points of intersection, and use hyperbolic geometry and cut and paste to construct a curve shorter than the systoles. \\

Unless otherwise stated, a set $C$ of curves will always be assumed have pairwise geometric intersection number at most 1. This is motivated by the fact that systoles also have this property.\\

It is known that for any fixed $l>0$, the number of curves on $S_{g}$ of length less than or equal to $l\in \mathbb{R}$ is finite, \cite{M}. This will be called \textit{local finiteness}. A point $x\in \mathrm{Sys}(C)$ therefore has a neighbourhood on which the systoles are all contained in the set $C$.\\

Another elementary property of strata is the following: If $\mathrm{Sys}(C'')$ is a stratum, and $\mathrm{Sys}(C')$ is on the boundary of $\mathrm{Sys}(C'')$, then $C''\subsetneq C'$. Loosely speaking, the dimension of $\mathrm{Sys}(C)$ goes down as the cardinality of $C$ goes up. If the curves in $C''$ fill $\mathcal{S}_{g}$, this is also true for any stratum on $\partial Sys(C'')$. It follows that $\mathcal{T}_{g}$ is closed.

\begin{definition}[Topological Morse function]
Let $M$ be a topological $n$-dimensional manifold. A continuous function $f:M\rightarrow \mathbb{R}$ is a topological Morse function if the points of $M$ are all either regular points or critical points. A regular point $p\in M$ is a point with an open neighbourhood $U$ in $M$, such that $U$ is a homeomorphic coordinate patch, with one of the coordinates on $U$ being the function $f$. For a critical point $p$, there is a neighbourhood $U$ of $p$, and an index, $k\in \mathbb{Z}$, $0\leq k\leq n$, such that $U$ is a homeomorphic coordinate patch with the coordinates $\{x_{1}, \ldots, x_{n}\}$, and in $U$, $f$ is given by the formula
\begin{equation*}
f(x)-f(p)=\sum_{i=1}^{i=n-k}x_{i}^{2} - \sum_{i-n-k+1}^{i=n}x_{i}^{2}
\end{equation*}
\end{definition}

An important example, \cite{Akrout}, of a topological Morse function is the systole function $f_{\mathrm{sys}}:\mathcal{T}_{g}\rightarrow \mathbb{R}_{+}$, that maps $x\in\mathcal{T}_{g}$ to the lengths of the systoles at $x$. A point at which there is only one systole is contained in an open stratum of dimension equal to the dimension of $\mathcal{T}_{g}$. It is not hard to show that all strata of dimension equal to the dimension of $\mathcal{T}_{g}$  are of this type. The systole function is smooth when restricted to a top-dimensional stratum, but in general it is only continuous.\\

A key result when discussing $\mathcal{P}_{g}$ is the following

\begin{theorem}[Proposition 1 of \cite{Thurston} and first proven in \cite{Bers}]
\label{Thurstonprop}
Suppose $C$ is a set of curves on $\mathcal{S}_{g}$ that do not fill. Then at any point $x$ of $\mathcal{T}_{g}$, there is a derivation in $T_{x}\mathcal{T}_{g}$ whose evaluation on every one of the functions $L(c_{i})$ for $c_{i}\in C$ is strictly positive. 
\end{theorem}

Theorem \ref{Thurstonprop} implies that the critical points of $f_{\mathrm{sys}}$ are all contained in $\mathcal{P}_{g}$. It might seem initially surprising that the systoles can intersect at all. This is a consequence of Theorem \ref{Thurstonprop} and the existence of an upper bound, Bers' constant, on $f_{\mathrm{sys}}$, \cite{Bersconstant}. As all maxima of $f_{\mathrm{sys}}$ are realised in the thick part of $\mathcal{T}_{g}$, which is a compact modulo the action of $\Gamma_{g}$, the existence of an upper bound for the $\Gamma_{g}$-equivariant function $f_{\mathrm{sys}}$ is a consequence of continuity.\\

The $\delta$-thick part of $\mathcal{T}_{g}$, $\mathcal{T}_{g}^{\delta}$ is defined to be the set of all points of $\mathcal{T}_{g}$ on which $f_{\mathrm{sys}}$ is less than or equal to $\delta$. The complement of the $\delta$-thick part of $\mathcal{T}_{g}$ is the $\delta$-thin part. When $\delta$ is less than a constant $\epsilon_{M}$ known as the Margulis constant, it is a consequence of the collar lemma, see for example Lemma 13.6 of \cite{FandM}, that the systoles in the $\delta$-thick part are pairwise disjoint.\\

Suppose $\delta$ is less than or equal to the Margulis constant. Denote by $\mathcal{C}_{g}^{\circ}$ the first barycentric subdivision of Harvey's curve complex $\mathcal{C}_{g}$. \\

\begin{theorem}[Proven in \cite{Ivanov}]
\label{complexembedding}
$\mathcal{C}_{g}^{\circ}$ is $\Gamma_{g}$-equivariantly homotopy equivalent to $\partial\mathcal{T}_{g}^{\delta}$.
\end{theorem}

The dimension of $\mathcal{C}_{g}$ is $3g-4$, which is quite a bit smaller than the dimension $6g-7$ of $\partial\mathcal{T}_{g}^{\delta}$. Despite this, in Section \ref{horizonsec}, a bordification of Schmutz Schaller's sets of minima (defined below) will be decribed in terms of subcomplexes of $\mathcal{C}_{g}$ or its barycentric subdivision $\mathcal{C}_{g}^{\circ}$. This is done by taking the intersection of the set of minima with $\partial \mathcal{T}_{g}^{\delta}$. In this intersection, the systoles are disjoint, so each set of systoles is a set of curves labelling a vertex of $\mathcal{C}_{g}^{\circ}$. As $\delta$ approaches zero, the induced stratification of the intersection of the set of minima with $\partial \mathcal{T}_{g}^{\delta}$ gives a subcomplex of $\mathcal{C}_{g}^{\circ}$.\\

When a metric is needed, it is often possible to assume any $\Gamma_{g}$-equivariant metric, such as the Weil-Petersson metric. The exception to this is when a metric $\mathbf{g}$, defined below, will be assumed when working with sets of minima.

\subsection{Loci and sets of minima}
\label{subminima}
This subsection will introduce Schmutz Schaller's sets of minima and state some of their basic properties. A reference is Section 2 of \cite{SchmutzMorse}. In addition, loci of sets of curves will be defined, and some basic properties from Section 3 of \cite{MS} will be surveyed.\\

A set $C$ of curves will be said to be \textit{eutactic} at $x\in \mathcal{T}_{g}$ if for every derivation $v\in T_{x}\mathcal{T}_{g}$ one of the following holds
\begin{itemize}
\item{$v(L(c_{i}))(x)>0$ and $v(L(c_{j}))(x)<0$ for some $c_{i}\neq c_{j}$ in $C$, or}
\item{$v(L(c_{i}))(x)=0$ for every $c_{i}$ in $C$.}
\end{itemize}

\begin{definition}[Set of minima $\mathrm{Min}(C)$ and $\partial\mathrm{Min}(C)$, \cite{SchmutzMorse}]
Fix a finite set $C$ of curves on $\mathcal{S}_{g}$. The set of minima, $\mathrm{Min}(C)$, is defined to be the set of all $x\in \mathcal{T}_{g}$ at which $C$ is eutactic. 
Alternatively, $\mathrm{Min}(C)$ can be defined to be the set of all points in $\mathcal{T}_{g}$ at which a length function $L(A,C)$ has its minimum for some $A\in \mathbb{R}_{+}^{|C|}$.
A point $x$ is in $\partial \mathrm{Min}(C)\cup \mathrm{Min}(C)$ iff there is no derivation $v$ in $T_{x}\mathcal{T}_{g}$ whose evaluation on $L(c)$ at $x$ is greater than zero for every $c$ in $C$.
\end{definition}

Equivalence of the two definitions of $\mathrm{Min}(C)$ is a consequence of the fact that a length function $L(A, C)$ is convex and hence has a minimum at $x$ iff the gradient at $x$ is zero.  

\begin{lemma}[Lemma 1 of \cite{SchmutzMorse}]
\label{lemma1}
When $C$ is a filling set of curves, for any fixed $A\in \mathbb{R}_{+}^{|C|}$ the length function $L(A, C)$ has a unique minimum in $\mathcal{T}_{g}$.
\end{lemma}

Lemma \ref{lemma1} gives a surjective map $\phi:\mathbb{R}_{+}^{|C|}\rightarrow \mathrm{Min}(C)$, that takes a point $A=(a_{1}, \ldots, a_{|C|})$ to the unique point in $\mathcal{T}_{g}$ at which $L(A,C)$ has its minimum. It follows from Lemma \ref{lemma1} and Theorem \ref{Thurstonprop} that $\mathrm{Min}(C)$ is nonempty iff $C$ fills.\\

The pre-image of $\phi$ is the intersection of an affine subspace with the positive orthant of $\mathbb{R}^{|C|}$. It has dimension at least 1 in $\mathbb{R}_{+}^{|C|}$. For $x\in \mathrm{Min}(C)$, the dimension of $\phi^{-1}(x)$ is equal to 
\begin{equation*}
 |C|-\mathrm{dim}(\mathrm{Span}\{\nabla L(c)(x)\ |\ c\in C\})
 \end{equation*}

It was shown in \cite{SchmutzMorse} that $\phi$ can be defined on a larger set whose image contains $\partial \mathrm{Min}(C)\cup \mathrm{Min}(C)$. The image of this set is called $\mathrm{Conv}(C)$ and consists of all the points in $\mathcal{T}_{g}$ at which there exists a convex function (not necessarily a length function) obtained as a linear combination of $\{\nabla L(c)\ |\ c\in C\}$. An admissible boundary point is a point in the pre-image of $\partial \mathrm{Min}(C)$ under $\phi$. Denote by $A$ the set of admissible boundary points of $\mathbb{R}_{+}^{|C|}$.  \\

\begin{definition}[$C$-regularity]
 A point $x$ in $\mathrm{Min}(C)$ is called $C$-regular if there is a neighbourhood $U$ of $x$ in $\mathrm{Conv}(C)$ on which the dimension of pre-images of points in $U$ under $\phi$ is constant. 
\end{definition}

A set of minima $\mathrm{Min}(C)$ is a cell in $\mathcal{T}_{g}$ unless the dimension of pre-image of points under $\phi$ varies, i.e. unless $C$-regularity breaks down on $\mathrm{Min}(C)$. For this reason, $\mathrm{Min}(C)$ will be described as a pinched cell; this is either a cell or a cell that has been collapsed along pairwise disjoint properly embedded submanifolds. The dimension of $\mathrm{Min}(C)$ at a point $x$ is defined to be the dimension of the span of $\{\nabla L(c)(x)\ |\ c\in C\}$.\\

\begin{lemma}[Corollary 13 of \cite{SchmutzMorse}]
\label{regularitylem}
When $C$-regularity holds at every point of $\mathrm{Min}(C)$, $\mathrm{Min}(C)$ is a continuously differentiable submanifold of $\mathcal{T}_{g}$.
\end{lemma}

A set of minima $\mathrm{Min}(C)$ not only has a boundary and a dimension, but as shown in Corollary 12 of \cite{SchmutzMorse}, it also has a tangent space at every point. This tangent space at $x\in \mathrm{Min}(C)$ is identified with the span of $\{\nabla L(c)(x)\ |\ c\in C\}$. 

\begin{lemma}
\label{themetric}
For every set of minima $\mathrm{Min}(C)$, there is a metric $\mathbf{g}$ on $\mathcal{T}_{g}$, possibly depending on $C$, for which the gradients $\{\nabla L(c)(x)\ |\ c\in C\}$ span the tangent space to $\mathrm{Min}(C)$ at every point $x$ in $\mathrm{Min}(C)$.
\end{lemma}
\begin{proof}
Once it is understood how to construct a suitable set of coordinates around points of $\mathrm{Min}(C)$ where $C$-regularity breaks down, a metric $\mathbf{g}$ can be constructed following the usual proof of the existence of a metric on a differentiable manifold using charts and partitions of unity. See for example \cite{Lee} for details.\\

Note that the orthogonal complement of $\{\nabla L(c)(x)\ |\ c\in C\}$, denoted $\{\nabla L(c)(x)\ |\ c\in C\}^{\perp}$, is independent of the choice of metric, as it can be characterised in terms of level sets of length functions.\\

\begin{figure}
\centering
\includegraphics[width=4cm]{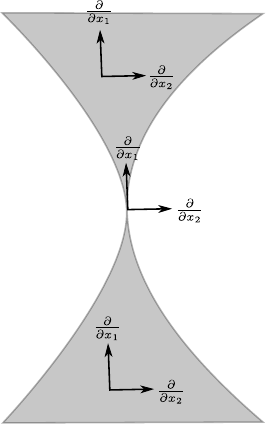}
\caption{A 2-dimensional illustration of how to construct coordinates near a point in $\mathrm{Min}(C)$ where regularity breaks down. The set of minima $\mathrm{Min}(C)$ is shaded.}
\label{applecore}
\end{figure}

A lower dimensional analogue of how to construct a set of coordinates around a point where regularity breaks down is illustrated in Figure \ref{applecore}. Suppose the dimension drops by $j>0$ at a point $x\in \mathrm{Min}(C)$ where regularity breaks down, and that the maximum dimension of $\mathrm{Min}(C)$ on a small neighbourhood $\mathcal{N}$ of $x$ is $n$. A chart on $\mathcal{N}$ and coordinates $x_{1}, \ldots, x_{n-j}, x_{n-j+1}, \ldots, x_{n}, x_{n+1}, \ldots, x_{6g-6}$ are chosen as follows: The vectors $\frac{\partial}{\partial x_{1}}, \ldots, \frac{\partial}{\partial x_{n}}$ span the tangent space to the image of $\mathcal{N}\subset \mathbb{R}^{6g-6}$ at points where the dimension is $n$. At points where the dimension drops, some or all of the vectors $\frac{\partial}{\partial x_{n-j+1}}, \ldots, \frac{\partial}{\partial x_{n}}$ are tangent to the image of $\{\nabla L(c)(x)\ |\ c\in C\}^{\perp}$ under the chart. The vectors $\frac{\partial}{\partial x_{n}}, \ldots, \frac{\partial}{\partial x_{6g-5}}$ are all tangent to the image of $\{\nabla L(c)(x)\ |\ c\in C\}^{\perp}$ under the chart.
\end{proof}

\begin{lemma}[Consequence of Lemma 14 and Theorem 15 of \cite{SchmutzMorse}]
\label{14}
If $x\in \mathcal{T}_{g}$ is in $\partial\mathrm{Min}(C)$, then $x\in \mathrm{Min}(C')$ for some $C'\subsetneq C$. If $C'$ is the largest subset of $C$ for which $x\in \mathrm{Min}(C')$ there exists a derivation $v\in T_{x}\mathcal{T}_{g}$ such that $vL(c)(x)=0$ for all $c\in C'$ and $vL(c)(x)<0$ for all $c\in C\setminus C'$.
\end{lemma}
\begin{remark}
\label{minfilling}
When $C$ is a filling set, but does not contain any proper subsets that fill, Lemma \ref{14} implies that $\mathrm{Min}(C)$ has empty boundary. Consequently, $\mathrm{Min}(C)$ cannot be pinched, and is therefore a continuously differentiable submanifold of $\mathcal{T}_{g}$ by Lemma \ref{regularitylem}. This is also a consequence of Lemmas 8 and 9 of \cite{SchmutzMorse}.
\end{remark}

In Lemma \ref{14}, the derivation $v$ is pointing into $\mathrm{Min}(C)$.\\

Some ideas that motivate the notion of Schmutz-Thurston duality are given in the next two theorems.

\begin{theorem}[Shown in \cite{SchmutzMorse} in the case of local maxima of $f_{\mathrm{sys}}$ and for general critical points in \cite{Akrout}]
\label{Akroutthm}
For any set $C'$, if $\mathrm{Min}(C')$ and $\mathrm{Sys}(C')$ intersect at all, the intersection is a unique point.
The point $p$ is a critical point of $f_{\mathrm{sys}}$ iff there exists a set $C$ such that $p$ is in $\mathrm{Min}(C)\cap\mathrm{Sys}(C)$.  
\end{theorem}



\begin{theorem}[Theorems 1.1 and 1.2 and Proposition 1.3 of \cite{MS}]
\label{MSthm}
There is a $\Gamma_{g}$-equivariant deformation retraction of $\mathcal{P}_{g}$ onto a complex $\mathcal{P}^{X}_{g}$ consisting of a choice of the unstable manifolds of the critical points of $f_{\mathrm{sys}}$. At a critical point $p\in \mathrm{Sys}(C)$, the tangent space to the unstable manifold is given by $\{\nabla L(c)(p)\ |\ c\in C\}^{\perp}$.
\end{theorem}
Examples studied in \cite{Ni} demonstrate that, for some values of $g$, $\mathcal{P}_{g}$ contains points not in $\mathcal{P}^{X}_{g}$.\\

\textbf{Loci.} Fix a set of curves $C=\{c_{1}, \ldots, c_{n}\}$ and a tuple $d=(d_{1}, \ldots, d_{n})$ in $\mathbb{R}^{n}$. The locus $E(C,d)$ is given by
\begin{equation*}
\{x\in \mathcal{T}_{g}\ |\ L(c_{i})(x)+d_{i}=L(c_{j})(x)+d_{j}\  \forall i,j\in 1,\ldots,n\}
\end{equation*}
The most important example of a locus has $d=(0, \ldots, 0)$, and will be written $E(C)$. As for $\mathcal{P}_{g}$, it follows from a theorem due to Lojasiewicz, \cite{Lojasiewicz1964}, that $E(C,d)$ can be triangulated. A stratum $\mathrm{Sys}(C)$ is a subset of $E(C)$.\\

For an object such as a stratum or locus that can be triangulated, tangent cones are defined in analogy with polyhedra. The \textit{tangent cone} to $E(C,d)$ at a point $x\in E(C,d)$ is the set of $v\in T_{x}\mathcal{T}_{g}$ such that $v=\dot{\gamma}(0)$ for a smooth oriented path with $\gamma(0)=x$ and $\gamma(t)$ contained in $E(C,d)$ for sufficiently small $t>0$. The tangent cone to a stratum is defined analogously. The tangent cone to $\mathcal{P}_{g}$ at $x\in \mathcal{P}_{g}$ is the union of the tangent cones of the strata in $\mathcal{P}_{g}$ incident on $x$.\\

A generalised minimum of the set $C$ of curves will now be defined. This is related to -- but should not be confused with -- a point in $\mathrm{Min}(C)$. For any $x\in \mathcal{T}_{g}$, there is a $d(x)$ such that $E(C,d(x))$ contains the point $x$. The point $x$ will be called a \textit{generalised minimum of} $C$ if $L(c)|_{E(C,d(x))}$ has a local minimum at $x$ for some (and hence every) $c\in C$.\\

A point in $\mathrm{Min}(C)$ is a generalised minimum of $C$, as are points of $\mathrm{Conv}(C)$. It is not the case that for a fixed set $C$, every minimum of $C$ is in either $\mathrm{Min}(C)$ or $\mathrm{Conv}(C)$. For example, when $C$ is a nonfilling set, both $\mathrm{Min}(C)$ and $\mathrm{Conv}(C)$ are empty. In contrast, when $C$ is the nonfilling set of curves in Figure \ref{nonfillingminima}, there are generalised minima of $C$. In this example, the generalised minima occur at points where the gradient of the length of the separating curve is linearly dependent on the gradients of the lengths of the nonseparating curves. \\

\begin{figure}
\centering
\includegraphics[width=4cm]{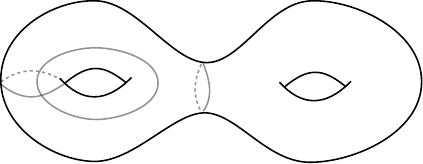}
\caption{When $C$ is the set of curves shown in grey, $C$ has generalised minima.}
\label{nonfillingminima}
\end{figure}


There are different ways of extending Lemma \ref{regularitylem} to the larger set consisting of the generalised minima of $C$. The next lemma is a simple statement along these lines. 

\begin{lemma}
\label{differentminima}
Suppose $C$ is a set of filling curves for which $\{\nabla L(c)\ |\ c\in C\}$ are linearly independent except at generalised minima of $C$, and that the dimension of the span of $\{\nabla L(c)\ |\ c\in C\}$ drops by one at the generalised minima of $C$. All the points realised as generalised minima of $C$ make up a smooth embedded submanifold of $\mathcal{T}_{g}$.
\end{lemma}
\begin{proof}
The loci $E(C,d)$ with varying $d$ are parameterised by the $|C|$-tuple $d$ whose entries sum to one. By assumption, the first $|C|-1$ entries are analytic functions $d_{1}, \ldots, d_{|C|-1}$ whose values on $\mathcal{T}_{g}$ are all regular and whose level sets intersect transversely. It follows from multiple applications of the pre-image lemma that loci $E(C,d)$ are smooth, embedded submanifolds without boundary. The same applies to level sets of $L(c)|_{E(C,d)}$ for every $c\in C$. The linear independence assumptions also imply that the only critical points of $L(c)|_{E(C,d)}$ for $c\in C$ are generalised minima; consequently each connected component of $E(C,d)$ has at most one generalised minimum of $C$. Since $C$ fills, it follows from the collar lemma that level sets of $L(c)|_{E(C,d)}$ for $c\in C$ are compact. Moreover, the level sets of $L(c)|_{E(C,d)}$ above the minima are codimension 1 spheres centered on the minimum.\\

Putting all these observations together, when $C$ fills, a generalised minimum of $C$ on $E(C,d)$ is a point. The functions $d_{1}, \ldots, d_{|C|-1}$ then give a local diffeomorphism from an open subset of $\mathbb{R}^{|C|-1}$ to a subset of $\mathcal{T}_{g}$ consisting of the generalised minima of $C$.
\end{proof}

The following is an extension of Lemma \ref{14} to generalised minima.

\begin{lemma}
\label{14generalised}
For $C'\subset C$, a generalised minimum of $C'$ is a generalised minimum of $C$.
\end{lemma}
\begin{proof}
Loci $E(C,d)$ for varying $d$ with $d|_{C'}=d'$ are contained in $E(C', d')$. A minimum $x$ of $C'$ in $E(C', d')$ is a point at which $v_{C'}$ is zero, and is therefore also a point at which the lengths of curves in $C$ cannot locally be decreased inside of the locus $E(C,d(x))$ passing through $x$.
\end{proof}

When discussing loci, the notation $v_{C}$ will always be used for a vector that gives a direction in which the lengths of curves in $C$ are all increasing at the same rate. Whenever possible, $v_{C}$ will be chosen to be nonzero and in the convex hull of $\{\nabla L(c)\ |\ c\in C\}$.\\

A locus $E(C,d)$ will be called 
\begin{itemize}
\item{balanced at $x\in E(C,d)$ if $x\in \mathrm{Min}(C)$ or there exists a nonzero $v_{C}(x)$ in the tangent cone to $E(C,d)$ at $x$ contained in the convex hull of $\{\nabla L(c)(x)\ |\ c\in C\}$.}
\item{semi-balanced at $x$ if it is not balanced but $x\in\mathrm{Min}(C')$ for some $C'\subset C$ or there is a nonzero $v_{C}(x)$ in the tangent cone to $E(C,d)$ on the boundary of the convex hull of $\{\nabla L(c)(x)\ |\ c\in C\}$.}
\item{unbalanced at $x$ if there does not exist a nonzero $v_{C}(x)$ in the tangent cone to $E(C,d)$ in or on the boundary of the convex hull of $\{\nabla L(c)(x)\ |\ c\in C\}$, or $x$ is an isolated generalised minimum of $C$ on $E(C,d)$ for which every neighbourhood of $x$ contains points in $E(C,d)$ that are unbalanced. }
\end{itemize}
It was shown in \cite{MS} that the property of being balanced, semi-balanced or unbalanced is independent of the choice of metric. A stratum $\mathrm{Sys}(C)$ is balanced (respectively semi-balanced or unbalanced) at $x$ if $E(C)$ is balanced (respectively semi-balanced or unbalanced) at $x$. \\

\begin{remark}
\label{balancedeverywhere}
Although the property of being balanced/semibalanced/unbalanced was defined for points of a locus $E(C,d)$, at least on connected open subsets of loci on which $\{\nabla L(c)\ |\ c\in C\}$ are linearly independent, these properties pass to loci, and hence to strata. In the case of balanced strata, this follows from an argument given in the proof of Proposition 3.13 of \cite{MS} and for unbalanced strata this was shown in the proof of Proposition 3.20 of \cite{MS}.
\end{remark}

\section{The horizon map}
\label{horizonsec}
Recall that for $\delta<\epsilon_{M}$, the boundary of the $\delta$-thick part of $\mathcal{T}_{g}$, $\partial\mathcal{T}_{g}^{\delta}$, is homotopy equivalent to Harvey's curve complex $\mathcal{C}_{g}$. The purpose of this section is to define the horizon map and determine its image. In the next section, this model of the thin part of $\mathrm{Min}(C)$ will be used to describe the structure of $\mathrm{Min}(C)$ under regularity assumptions.\\

Throughout this section, $C$ will be assumed to be a filling set of curves with pairwise intersection at most one.\\

In what follows, it is tempting to think of objects such as the horizon map defined below as having image in a dual complex to $\mathcal{C}_{g}$. However, as $\mathcal{C}_{g}$ is not locally finite, such an object is not a simplicial complex. For this reason, the image of the horizon map will be defined to be in the barycentric subdivision $\mathcal{C}^{\circ}_{g}$ of $\mathcal{C}_{g}$. The simplicial complex $\mathcal{C}^{\circ}_{g}$ has vertices labelled by multicurves $m_{1}<\ldots < m_{n}$, where the definition is made that $m_{i}<m_{j}$ if $m_{i}$ is a subset of $m_{j}$. The interior of the simplex is labelled by the smallest multicurve in this ordering. \\

\begin{definition}[Horizon map $h:\mathrm{Min}(C)\rightarrow \mathcal{C}^{\circ}_{g}$]
The horizon map $h:\mathrm{Min}(C)\rightarrow \mathcal{C}^{\circ}_{g}$ takes the set of minima $\mathrm{Min}(C)$ to the subcomplex of $\mathcal{C}^{\circ}_{g}$ with vertices labelled by multicurves on $\mathcal{S}_{g}$ that become arbitrarily short on $\mathrm{Min}(C)$.
\end{definition}

\begin{remark}
Although a set of minima $\mathrm{Min}(C)$ might not be uniquely labelled by the set of curves $C$, it will follow that the horizon map does not depend on the choice of $C$ but only on the set of minima. This is illustrated in Example \ref{schmutzexample}.
\end{remark}

Recall from Lemma \ref{14} that a point $x\in \partial \mathrm{Min}(C)$ is contained in $\mathrm{Min}(C')$ for some filling subset $C'\subset C$. In what follows, Lemma \ref{14} will be generalised by defining sets of minima in the Weil-Petersson metric completion $\overline{\mathcal{T}}_{g}$ of $\mathcal{T}_{g}$ for nonfilling subsets of $C$, and showing how these sets of minima for nonfilling subsets describe the intersection of $\mathrm{Min}(C)$ with the thin part of $\mathcal{T}_{g}$.These should not be confused with the generalised minima for nonfilling sets defined earlier, which are contained in $\mathcal{T}_{g}$.\\

Recall that by Theorem \ref{Thurstonprop}, if $C'\subset C$ does not fill, $\mathrm{Min}(C)$ is empty. However, for any fixed $A\in \mathbb{R}_{+}^{|C'|}$, it is still possible to find sequences of points in $\mathcal{T}_{g}$ on which $L(A, C')$ approaches its infimum.

\begin{definition}[$\mathrm{Min}(C')^{\infty}$ for nonfilling $C'$]
$\mathrm{Min}(C')^{\infty}$ is the set of all limits in $\overline{\mathcal{T}}_{g}$ of sequences $\{x_{i}\}$ in $\mathcal{T}_{g}$ for which there exists an $A\in \mathbb{R}_{+}^{|C'|}$ such that $\lim_{i\rightarrow \infty}L(A,C')(x_{i})$ approaches its infimum on $\mathcal{T}_{g}$.
\end{definition}

Choose and fix a nonfilling $C'\subset C$. Suppose $A$ is a $|C|$-tuple with all entries positive except for the entries corresponding to curves in $C\setminus C'$, which are zero. Let $A_{i}$, $i=1, 2, \ldots$ be a sequence of $|C|$-tuples with strictly positive entries, such that $A_{i}\rightarrow A$ as $i\rightarrow \infty$. For each value of $i$ there is a point $x_{i}$ in $\mathrm{Min}(C)$ at which the infimum of $L(A_{i}, C)$ is realised. As $i\rightarrow \infty$ the infimum of $L(A, C)$ is realised as the limit of $L(A_{i}, C)(x_{i})$. The set $\mathrm{Min}(C')^{\infty}$ will be said to be \textit{adherent} to $\mathrm{Min}(C)^{\infty}$ (in case of nonfilling $C$) or adherent to $\mathrm{Min}(C)$ (in case of filling $C$).\\

A sequence $C_{1}\subsetneq C_{2} \subsetneq \ldots C_{k}$ of nonfilling subsets of $C$ will be called \textit{maximal} if no subset can be inserted to obtain a longer sequence. For a nonfilling $C'\subset C$, $m(C')$ will be used to denote the multicurve consisting of the set of curves on the boundary of the subsurface of $\mathcal{S}_{g}$ filled by $C'$. Note that curves are defined up to isotopy; while a given isotopy class might be realised more than once on the boundary of a subsurface filled by $C'$, each isotopy class has at most one representative in $m(C')$. For this reason, $m(C')$ might be nonseparating.\\

The next corollary determines the image of $\mathrm{Min}(C)$ under the horizon map.

\begin{corollary}
\label{defncor1}
Suppose $C_{1}\subsetneq C_{2} \subsetneq \ldots C_{n-1} \subsetneq C_{k}$ is maximal. It follows that
\begin{enumerate}
\item{Every point in the Weil-Petersson metric completion of $\mathrm{Min}(C)$ is either in $\mathrm{Min}(C')$ for some filling $C'\subset C$ or in $\mathrm{Min}(C')^{\infty}$ for some nonfilling $C'\subset C$.}
\item{$\mathrm{Min}(C_{k})^{\infty}\cap \overline{\mathrm{Min}(C)}\subset \mathrm{Sys}(m(C_{k}))^{\infty}$}
\item{The multicurve $m(C_{i})$ is pinched on $\mathrm{Min}(C_{i})^{\infty}$ for every $i=1, \ldots, k$}
\item{$\mathrm{Min}(C_{i})^{\infty}$ is adherent to $\mathrm{Min}(C')$ or $\mathrm{Min}(C')^{\infty}$ for any $C'$ that contains $C_{i}$.}
\item{The sequence $C_{1}\subsetneq C_{2} \subsetneq \ldots C_{k-1} \subsetneq C_{k}$ determines a simplex in the image of $\mathrm{Min}(C)$ under the horizon map.}
\end{enumerate}
\end{corollary}
\begin{proof}
Part (1) follows from Lemma \ref{14} and the definitions. \\

To prove parts (2) and (3), Riera's theorem, \cite{Riera}, implies that the Weil-Petersson gradient of $L(c)$, for every $c\in m(C_{k})$, is a direction in which the length of every curve in $C_{k}$ is strictly increasing. It follows from the collar lemma that no curve intersecting the subsurface filled by $C_{k}$ is pinched in $\mathrm{Min}(C_{k})^{\infty}$. Once $m(C_{k})$ is pinched, pinching any other curves disjoint from the subsurface filled by $C_{k}$ does not shorten the curves in $C_{k}$ any further, and the assumption of maximality implies it lengthens one or more curves in $C\setminus C_{k}$. Consequently, such noded surfaces do not represent points in $\mathrm{Min}(C_{k})^{\infty}\cap \overline{\mathrm{Min}(C)}$. \\
Part (4) also follows from the definitions.\\

For part (5), the sequence $C_{1}\subsetneq C_{2} \subsetneq \ldots C_{k-1} \subsetneq C_{k}$ determines the simplex in $\mathcal{C}_{g}^{\circ}$ with vertices labelled by $m(C_{1}), m(C_{2}), \ldots, m(C_{k-1})$; no claim is being made that these multicurves are all distinct. Note that two different sequences of subsets could determine the same simplex.
\end{proof}

For a set of minima $\mathrm{Min}(C)$, define $C_{\mathrm{max}}$ to be the largest set of curves containing $C$ for which every nonfilling subset $C'$ of $C_{\mathrm{max}}$ gives $\mathrm{Min}(C')^{\infty}$ adherent to $\mathrm{Min}(C)$. $C_{\mathrm{max}}$ is not assumed to have the property that curves in this set intersect pairwise at most once.\\

Intuitively, a set of minima behaves like a convex hull of the minima on its boundary. As the horizon map determines this boundary, this suggests the next conjecture. 

 \begin{conjecture}
\label{uniquelydetermines}
The image of the horizon map uniquely determines the set of minima $\mathrm{Min}(C)=\mathrm{Min}(C_{\mathrm{max}})$.
\end{conjecture}

\section{Faces and strata}
\label{twocases}
A combinatorial model for a set of minima $\mathrm{Min}(C)$ containing a critical point $p\in \mathrm{Sys}(C)$ will be discussed. Under regularity assumptions, $\mathrm{Min}(C)$ will be described as the image of a map $\psi$ from a polytope with vertices corresponding to the vectors $\{-\nabla L(c)(p)\ |\ c\in C\}$, and the faces of $\mathrm{Min}(C)$ will be related to strata in $\mathcal{P}_{g}$ adjacent to $p$. The key idea is the recurring theme that the decomposition of the tangent space $T_{p}\mathrm{Min}(C)$ at a critical point $p\in \mathrm{Sys}(C)$ given by a Voronoi-like decomposition $\mathcal{V}(p)$ gives a surprising accurate approximation of the global combinatorial structure of $\mathrm{Min}(C)$. \\

\subsection{Petals}
\label{corollaries}

Much of the combinatorial information about sets of minima comes from studying the top-dimensional strata in $\mathcal{T}_{g}$ adjacent to $p$; the ``petals''. A number of useful corollaries and basic properties of petals are presented in this subsection.\\

A minimal filling set of curves is a set of curves, intersecting pairwise at most once, that does not contain any filling subsets. When $C$ is a minimal filling set, for every $c\in C$ there is a curve $c^{*}$ with the property that $c^{*}$ intersects $c$ but is disjoint from every curve in $C\setminus \{c\}$.\\

Many results about petals are based on the following lemma. 

\begin{lemma}[Proposition 3.14 of \cite{MS}]
\label{mycor}
Suppose $c_{1}$ and $c_{2}$ are curves on $\mathcal{S}_{g}$ that intersect at most once, where $c_{1}$ is nonseparating. Suppose also that $v_{1,2}(x)$ denotes a vector in $T_{p}\mathcal{T}_{g}$ representing a direction in which $L(c_{1})$ and $L(c_{2})$ are increasing at the same rate. Then at every point $x$ for which $L(c_{1})(x)=L(c_{2})(x)$ there is a vector $v_{1,2}(x)\neq 0$, contained in the interior of the convex hull of $\nabla L(c_{1})(x)$ and $\nabla L(c_{2})(x)$.\\

Alternatively, the lemma holds with the set $\{c_{1}, c_{2}\}$ replaced by a subset of a minimal filling set of curves, or a multicurve.
\end{lemma}

Note that a multicurve could contain separating curves, so is not necessarily contained in a minimal filling set.

\begin{remark}
It follows from Proposition 3.19 of \cite{MS} that Lemma \ref{mycor} holds for any choice of $\Gamma_{g}$-equivariant metric used to define the gradients.
\end{remark}

\begin{lemma}
\label{subset}
Suppose $c_{1}$, $c_{2}$ and $c_{3}$ are nonseparating, pairwise nonhomotopic, intersecting pairwise at most once and not contained in a genus 1 surface with 1 boundary component. Then $\{c_{1}, c_{2}, c_{3}\}$ is contained in a minimal filling set of curves.
\end{lemma}
\begin{proof}
The property of being contained within a minimal filling set of curves passes to $\Gamma_{g}$-orbits of sets of curves. It follows from the classification of surfaces that for each genus there are only finitely many orbits to consider. In genus 2 and 3 it is possible construct examples that prove the lemma in all of the finitely many different orbits. These examples generalise in a straightforward way to higher genus. The details are left to the reader. 
\end{proof}

Lemma \ref{mycor} will be used to show that $L(c_{2})$ is not decreasing as fast as $L(c_{1})$ in the direction of $-\nabla L(c_{1})$ at points close to where $L(c_{1})=L(c_{2})$. This has a couple of consequences, which will be stated as corollaries.

\begin{corollary}
\label{staysin}
Suppose $C$ is a filling set, and $c_{1}\in C$. Denote by $\gamma:[0,T]\rightarrow \mathcal{T}_{g}$ a smooth path with $\gamma(0)$ in $\mathrm{Sys}(\{c_{1}\})\cap \mathrm{Min}(C)$ and $\dot{\gamma}(t)=-\nabla L(c_{1})$ for all $t\in (0,T)$, where gradients are defined with respect to the metric $\mathbf{g}$ from Lemma \ref{themetric}. The path $\gamma$ is contained in $\mathrm{Sys}(\{c_{1}\})\cap \mathrm{Min}(C)$
\end{corollary}
\begin{proof}
The path $\gamma$ cannot leave $\mathrm{Sys}(\{c_{1}\})$, because if it did, it would have to pass through a point $\gamma(t)$ at which $L(c_{1})(\gamma(t))=L(c_{2})(\gamma(t))$ for a curve $c_{2}$ satisfying the conditions of Lemma \ref{mycor}. However, this is not possible, because $L(c_{1})(\gamma(0))<L(c_{2})(\gamma(0))$ and by lemma \ref{mycor}, $L(c_{1})$ is decreasing faster along $\gamma$ than $L(c_{2})$ near where equality occurs.\\

It follows from Lemma \ref{14} that $\gamma$ cannot leave $\mathrm{Min}(C)$, because the tangent vector to $\gamma$ is always in the tangent cone to $\mathrm{Min}(C)$.
\end{proof}

It is a consequence of local finiteness and the definition of $\mathrm{Min}(C)$ that on a neighbourhood of a critical point $p\in \mathrm{Sys}(C)\cap \mathrm{Min}(C)$, $f_{\mathrm{sys}}$ is decreasing in $\mathrm{Min}(C)$ radially away from $p$, where the dimension of $\mathrm{Min}(C)$ at $p$ is equal to the index of $p$. In the case of a local maximum, the index of $p$ is equal to the dimension of $\mathcal{T}_{g}$.\\

Denote by $\mathcal{B}$ a ball in $\mathrm{Min}(C)$ centered on $p$. This ball is chosen small enough so that local finiteness ensures that the systoles at every point within it are contained in $C$, and $f_{\mathrm{sys}}$ is decreasing radially away from $p$ within $\mathcal{B}$.\\

The set $\mathrm{Min}(C)$ contains what will be called \textit{petals}. Each petal corresponds to a curve $c\in C$, and consists of all the points in the intersection $\mathrm{Sys}(\{c\})\cap \mathcal{B}$ as well as the images of all paths of the form $\gamma: [0,a]\rightarrow \mathrm{Min}(C)\cap Sys(\{c\})$, $a\in \mathbb{R}$ with $\gamma(0)\in Sys(\{c\})\cap \mathrm{Min}(C)\cap \mathcal{B}$ and $\dot{\gamma}=-\nabla L(c)$ from Corollary \ref{staysin}. Since gradients are defined here with respect to $\mathbf{g}$, a petal labelled by $c\in C$ is contained in $\mathrm{Sys}(\{c\})\cap \mathrm{Min}(C)$ by Corollary \ref{staysin}.

\begin{corollary}
\label{connectedcor}
$\mathrm{Sys}(\{c\})$ is connected.
\end{corollary}
\begin{proof}
Construct a path $\gamma$ as in Corollary \ref{staysin}. In the limit as $t\rightarrow \infty$, $\gamma(t)$ approaches a point in $\mathrm{Min}(\{c\})^{\infty}$ in the metric completion of $\mathcal{T}_{g}$ with respect to the Weil-Petersson metric. The corollary then follows from the fact that $\mathrm{Min}(\{c\})^{\infty}$ is connected.
\end{proof}

\begin{corollary}[Corollary 3.18 of \cite{MS}]
\label{flowercor}
$\mathrm{Sys}(C)$ is on the boundary of $\mathrm{Sys}(\{c\})$ for every $c\in C$.
\end{corollary}
\begin{proof}
It follows from local finiteness and Lemma \ref{mycor} that for a point $x$ in $\mathrm{Sys}(C)$,  $-\nabla L(c)(x)$ is in the tangent cone to $\mathrm{Sys}(\{c\})$ for every $c\in C$.
\end{proof}

Recall from Corollary \ref{defncor1} that $\mathrm{Min}(\{c\})^{\infty}$ is adherent to $\mathrm{Min}(C)$ for every $c\in C$. Since $L(c)$ is a convex function with respect to the Weil-Petersson metric, $L(c)$ increases with Weil-Petersson distance from $\mathrm{Min}(\{c\})^{\infty}$. There are estimates due to Wolpert, for example in \cite{Wo3}, of how fast $L(c)$ increases with Weil-Petersson distance from $\mathrm{Min}(\{c\})^{\infty}$. Such estimates suggest that a decomposition of $\mathrm{Min}(C)$ into regions on which each of the curves in $C$ are shortest is analogous to a Voronoi decomposition of $\mathrm{Min}(C)$ centered around $\{\mathrm{Min}(\{c\})^{\infty}\ |\  c\in C\}$. This is illustrated in Figure \ref{starlike}.\\

\begin{figure}
\centering
\includegraphics[width=6cm]{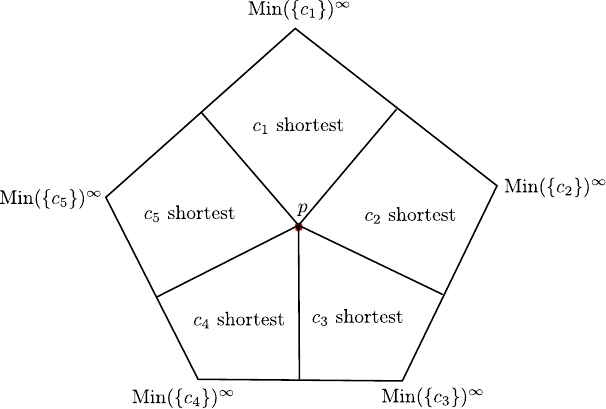}
\caption{A cartoon version of $\mathrm{Min}(C)$, showing how the decomposition into regions on which each of the curves in $C$ are the shortest elements of the set $C$ behaves like a Voronoi decomposition of $\mathrm{Min}(C)$ centered around $\{\mathrm{Min}(\{c\})^{\infty}\ |\  c\in C\}$.}
\label{starlike}
\end{figure}

The next proposition will be used in the proof of Lemma \ref{neededforduality}. It gives a geometric condition that can be used to characterise a subset of a set $C$ that fills the same subsurface of $\mathcal{S}_{g}$ as $C$.

\begin{proposition}
\label{29}
Suppose $C$ is a set of curves and $x\in \mathcal{T}_{g}$ is not a generalised minimum of $C$. Suppose also that a nonzero $v_{C}(x)$ exists that cannot be chosen to be in the convex hull of $\{\nabla L(c)(x)\ |\ c\in C\}$. Then the vertices of a face $f$ of the convex hull of $\{\nabla L(c)(x)\ |\ c\in C\}$ closest to $v_{C}$ are the gradients of lengths of curves in a set $C(f)$ that fill the same subsurface of $\mathcal{S}_{g}$ as $C$.
\end{proposition}
\begin{proof}
When $C$ is a filling set, this is Proposition 3.20 of \cite{MS}. When $C$ does not fill, the proof of Proposition 3.20 of \cite{MS} also applies with $\mathrm{Min}(C(f))$ replaced by $\mathrm{Min}(C(f))^{\infty}$ and by observing that the length of $v_{C(f)}$ approaches zero near $\mathrm{Min}(C(f))^{\infty}$ by an estimate derived from Riera's formula given in \cite{Wo3}.
\end{proof}

Some polytopes will now be defined. These polytopes will be used to encode adjacency information between petals in the model for $\mathrm{Min}(C)$ that will be discussed in the next subsection.

\begin{definition}[$\mathcal{V}(x)$ and $\mathcal{D}(x)$]
Suppose $C$ is a filling set, and $x\in \mathrm{Sys}(C)$. A ``Voronoi-like'' cell decomposition $\mathcal{V}(x)$ of the unit tangent space to $\mathcal{T}_{g}$ at $x$ has top-dimensional cells labelled by different curves in $C$. The interior of the cell labelled by $c\in C$ consists of all the unit vectors in $T_{x}\mathcal{T}_{g}$ for which the inner product with $-\nabla L(c)$ is greater than the inner product with $-\nabla L(c')$ for every $c'\in C\setminus \{c\}$. Lower dimensional cells of $\mathcal{V}(x)$ are labelled by subsets of $C$, where $v\in T_{x}\mathcal{T}_{g}$ is in the cell labelled by $C'$ if the inner product of $v$ with every element of $\{-\nabla L(c)\ |\ c\in C'\}$ is equal and greater than the inner product of every element of $\{-\nabla L(c)\ |\ c\in C\setminus C'\}$.\\

A ``Delaunay-like'' cell decomposition $\mathcal{D}(x)$ of the unit tangent space to $\mathcal{T}_{g}$ at $x$ is the dual to $\mathcal{D}(x)$. The vertices of $\mathcal{D}(x)$ are labelled by curves in $C$.\\
\end{definition}

Although $\mathcal{V}(x)$ and $\mathcal{D}(x)$ are topological duals, they might not always be geometric duals. When not, it can be shown that this implies the existence of unbalanced or semi-balanced strata adjacent to $x$.


\begin{remark}[Cells of $\mathcal{V}(p)$ and strata adjacent to a local maximum $p\in \mathrm{Sys}(C)$ of $f_{\mathrm{sys}}$]
\label{correspondencerem}
Suppose $\{c_{1}, c_{2}\}\subset C$. The vector fields $\nabla L(c_{1})$ and $\nabla L(c_{2})$ on $\mathcal{T}_{g}$ are everywhere linearly independent as a consequence of Lemma 3.3 of \cite{MS}. Linear independence together with the pre-image lemma guarantee the existence of an embedded submanifold $E(\{c_{1}, c_{2}\})$ containing $p$. The codimension of $E(\{c_{1}, c_{2}\})$ is one. Consequently, $\{c_{1}, c_{2}\}\subset C$ is the label of a stratum adjacent to $p$ iff $\{c_{1}, c_{2}\}$ is the label of a cell of $\mathcal{V}(p)$. An analogous statement holds when $\{c_{1}, c_{2}\}$ is replaced by any subset $C'$ of $C$ for which the gradients of the lengths of curves are linearly independent at $p$. \\

When the vertices of a cell of $\mathcal{V}(p)$ do not represent linearly independent vectors, a cell of $\mathbf{\mathcal{D}(p)}$ labelled by a subset $C'$ of $C$ might not determine a stratum, because equality of curve lengths might only be achievable to first order.\\

It is known from the proof of Theorem 4.5 of \cite{MS} that every vector in the cell of $\mathcal{V}(p)$ is in the tangent cone to a stratum labelled by a set of curves that fill the same subsurface of $\mathcal{S}_{g}$ as the set of curves labelling the cell. For sufficiently small $\epsilon$, a cell decomposition of a codimension 1 sphere of radius $\epsilon$ centered on $p$ inherits a cell decomposition from the stratification. This cell decomposition is very similar to the cell decomposition $\mathcal{V}(p)$ of the unit tangent bundle, the only potential difference is that some cells of the former may be collapsed in the limit as $\epsilon$ approaches zero.
\end{remark}


\subsection{Polytopal regularity, folding and equivalence classes}
\label{backgroundforthm}

The concepts needed for Theorem \ref{onlyfaces} will now be explained and motivated one by one. \\

\textbf{Polytopal Regularity.} Recall from Lemma \ref{14} that a point in $\mathcal{T}_{g}$ on $\partial \mathrm{Min}(C)$ is in $\mathrm{Min}(C')$ for some $C'\subsetneq C$. \\

\begin{theorem}[Theorem 15 of \cite{SchmutzMorse}]
\label{15}
Suppose $x\in \mathcal{T}_{g}$ is in $\mathrm{Min}(C'')\subset \overline{\mathrm{Min}(C)}$ and $C''\subset C$ is not contained in any other subset of $C$ satisfying this condition. Define 
\begin{equation*}
\mathcal{G}=\{C\}\cup\{C'\ |\ C''\subset C'\subset C \text{ and }\exists v\in T_{x}\mathcal{T}_{g}\text{ such that }
vL(c)(x)=0 
\end{equation*}
\begin{equation*}
\forall c\in C'\text{ and }vL(c)(x)>0 \ \forall c\in C\setminus C'\}
\end{equation*}
If $x$ is $C'$-regular for every $C'\in \mathcal{G}$, there exists a $C$-regular neighbourhood $U$ of $x$ in $\mathrm{Conv}(C)$ such that, in $U$
\begin{equation*}
\mathrm{Min}(C)\cup\partial \mathrm{Min}(C)=\cup_{C'\in \mathcal{G}}\mathrm{Min}(C')
\end{equation*}
where this is a disjoint union.
\end{theorem}

Theorem \ref{15} is used to ensure that $\mathrm{Min}(C)$ behaves like a polytope, with faces given by sets of minima of subsets of $C$. Even when $\mathrm{Min}(C)$ is pinched, the author is not aware of any examples in which regularity breaks down so badly that it is not possible to make sense of the notion of ``faces'' of $\mathrm{Min}(C)$. Perhaps such examples exist in Teichm\"uller spaces of large genus surfaces. \\

The set of minima $\mathrm{Min}(C)$ will be said to satisfy \textit{polytopal regularity} if the conditions of Theorem \ref{15} are satisfied at every point of $\overline{\mathrm{Min}(C)}$. In addition, for every filling $C'\subset C$, the dimension of the span of $\{\nabla L(c)\ |\ c\in C'\}$ is required to be constant on $\mathrm{Min}(C)$ away from generalised minima of $C'$, and constant over the set of generalised minima of $C'$.\\

\textbf{Equivalence classes of strata.} For a critical point $p\in \mathrm{Sys}(C)$, it can happen that the number of elements of $C$ is much larger than the dimension of $\mathrm{Min}(C)$. Unlike strata, sets of minima are not uniquely labelled by sets of curves. An example is given by the set of curves in Figure \ref{genus2example1}. For this reason, an equivalence relation $\sim$ on strata is defined, where $\mathrm{Sys}(C_{1})\sim \mathrm{Sys}(C_{2})$ if $\mathrm{Min}(C_{1})=\mathrm{Min}(C_{2})$.\\

\begin{example}
\label{schmutzexample}
Of the six curves shown in Figure \ref{genus2example1}, any filling subset determines the same set of minima. Note that any filling subset gives the same horizon map. It is also known \cite{SchmutzMorse} that this set of minima is a cell of dimension 3 on the boundary of the set of minima labelled by the set of systoles of the Bolza surface. As discussed in the last section of \cite{MS}, every filling subset of this set of curves determines a stratum adjacent to the maximum of $f_{\mathrm{sys}}$ given by the Bolza surface.

\begin{figure}
\centering
\includegraphics[width=8cm]{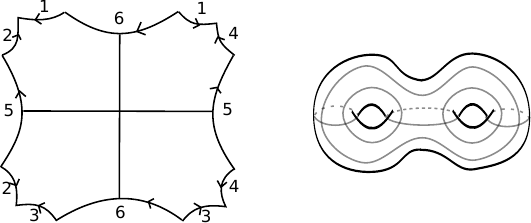}
\caption{This figure is taken from \cite{MS}, and the example was first mentioned in \cite{SchmutzMorse}. The edges on the fundamental domain on the left lie along the six curves shown on the right. The numbers on the edges of the fundamental domain show the edge pairings. A critical point of $f_{\mathrm{sys}}$ occurs when the six curves shown have the same length and intersect in right angles.}
\label{genus2example1}
\end{figure}

\end{example}

\textbf{Folding.} Recall from Remark \ref{correspondencerem} that $\mathcal{V}(p)$ determines the set of strata adjacent to a critical point $p$, except when a cell is labelled by a set of curves $C'$ whose gradients at $p$ are linearly dependent. In the linear dependent case, it is possible that equality of the lengths of curves in $C'$ only holds to first order in the direction of $v_{C'}(p)$. In this case, there is a point $x$ near $p$, the plane containing the tips of the vectors $\{-\nabla L(c)(p)\ | \ c\in C'\}$ needs to be folded to contain the tips of the vectors $\{-\nabla L(c)(x)\ | \ c\in C'\}$. This is illustrated in Figure \ref{folding1} on the left. The right of Figure \ref{folding1} shows the two resulting strata with the same tangent cone at $p$.\\

\begin{figure}
\centering
\includegraphics[width=10cm]{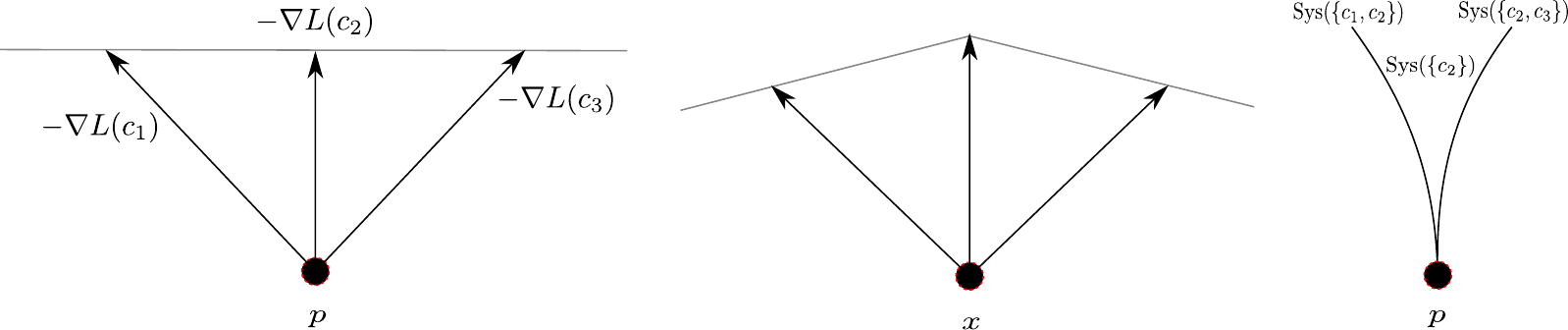}
\caption{The left part of the figure shows a set of linearly dependent gradients labelling a cell of $\mathcal{V}(p)$. At a point $x$ near $p$, these gradients are shown in the middle, and the corresponding strata are shown on the right.}
\label{folding1}
\end{figure}

At $p$, one might have linearly dependent gradients as in Figure \ref{folding1} that describe a facet of $\mathcal{D}(p)$. The set of curves labelling this facet however split into subsets that label different strata with the same tangent cone at $p$. Suppose $C_{1}$ and $C_{2}$ are subsets of $C'$ that label 1-dimensional strata adjacent to $p$. In such a case, the face of $\mathcal{D}(p)$ labelled by $C'$ will be said to be \textit{folded} along $C_{1}\cap C_{2}$, with \textit{folded faces} labelled by $C_{1}$ and $C_{2}$.\\

The point $p$ cannot be a minimum of $C_{1}$ or $C_{2}$ because the vectors $-v_{C_{1}}(p)$ and $-v_{C_{2}}(p)$ are in the tangent cones to the strata $\mathrm{Sys}(C_{1})$ and $\mathrm{Sys}(C_{2})$ respectively. Since $\mathrm{Sys}(\{c_{1}\})$ and $\mathrm{Sys}(\{c_{2}\})$ are 1-dimensional at $p$, polytopal regularity then ensures that $\{\nabla L(c)(x)\ | \ c\in C_{1}\}$ spans $T_{x}\mathcal{T}_{g}$ at every point of $\mathrm{Min}(C)$ that is not a minimum of $C_{1}$. Similarly for $C_{2}$. However, the point $p$ \textit{is} a minimum of $C_{1}\cup C_{2}$ or of $C'$. When assuming polytopal regularity, it follows that at a minimum $x$ of $C'$ in $\mathrm{Min}(C)$, $\{\nabla L(c)(x)\ | \ c\in C'\}$ spans $T_{x}\mathcal{T}_{g}$, so $\mathrm{Min}(C')$ has dimension equal to that of $\mathcal{T}_{g}$ or $\mathrm{Min}(C)$. It will be shown that the set of minima $\mathrm{Min}(C')$ has a pair of facets, $\mathrm{Min}(C_{1})$ and $\mathrm{Min}(C_{2})$ with common face $\mathrm{Min}(C_{1}\cap C_{2})$. \\

\textbf{The map }$\boldsymbol{\psi}$. Suppose $p$ is a local maximum of $f_{\mathrm{sys}}$. Then $T_{p}\mathrm{Min}(C)\simeq \mathbb{R}^{5g-6}$. By abuse of notation, $\mathcal{D}(p)$ will also be used to refer to a polytope in $\mathbb{R}^{6g-6}\simeq T_{p}\mathrm{Min}(C)$ with vertices given by the negative gradients $\{-\nabla L(c)(p)\ |\ c\in C\}$. Assuming $\mathrm{Min}(C)$ satisfies polytopal regularity, a map $\psi$ from the polytope $\mathcal{D}(p)$ to a bordification of $\mathrm{Min}(C)$ will be constructed. It will follow from Theorem \ref{onlyfaces} that the map $\psi$ is a homeomorphism in the interior of $\mathcal{D}(p)$.\\

First of all, $\psi(0)=p$. To define the map $\psi$ on other points of $\mathcal{D}(p)$, specific choices will be made for writing vectors as linear combinations of the vectors making up the vertices, and flowlines constructed.\\

Every point in a facet of $\mathcal{D}(p)$ with vertices $\{-\nabla L(c)(p)\ |\ c\in C'\subsetneq C\}$ can be written as a nonnegative linear combination of $\{-\nabla L(c)(p)\ |\ c\in C'\subsetneq C\}$. When the vertices $\{-\nabla L(c)(p)\ |\ c\in C'\subsetneq C\}$ are not linearly independent, make a continuous choice of how to write vectors as linear combinations of the vertices. When the facet is folded along $C''\subset C'$, choices are made that respect the folding. This means that the convex hull of $\{-\nabla L(c)(p)\ |\ c\in C'\}$ is cut into two pieces along the convex hull of $\{-\nabla L(c)(p)\ |\ c\in C''\}$; suppose one piece is the convex hull of $\{-\nabla L(c)(p)\ |\ c\in C_{1}\subset C'\}$ and the other piece is the convex hull of $\{-\nabla L(c)(p)\ |\ c\in C_{2}\subset C'\}$. Then a vector in $\mathcal{D}(p)$ in the convex hull of $\{-\nabla L(c)(p)\ |\ c\in C_{1}\subset C'\}$ is written as a nonnegative linear combination of $\{-\nabla L(c)(p)\ |\ c\in C_{1}\subset C'\}$. Similarly for $C''$ and $C_{2}$.\\

Every vector in the interior of the facet labelled by $C'$ is the negative gradient of a length function $L(A,C')$ for some $A$ with all entries nonnegative.\\

\textbf{Flowcones.} A \textit{flowcone} labelled by $C'\subset C$ corresponding to a folded face of a facet of $\mathcal{D}(p)$ consists of all the points in the image of paths of the form $\alpha:\mathbb{R}_{\geq 0}\rightarrow \mathrm{Min}(C)$ with $\alpha(0)=p$ and $\dot{\alpha}(t)=-\nabla L(A,C')$. Here $A$ is a tuple with the property that one of the vectors in the facet labelled by $C'$ is written as the linear combination of $\{-\nabla L(c)(p)\ |\ c\in C'\}$ given by $-\nabla L(A,C')(p)$. Let $C''\subset C'$ be the subset of $C'$ for which the corresponding entries in $A$ are nonzero. The endpoint of the path $\alpha$ is either in $\mathrm{Min}(C'')$ if $C''$ fills, or in $\mathrm{Min}(C'')^{\infty}$ if it does not.\\

As the metric $\mathbf{g}$ from Lemma \ref{themetric} is assumed, it follows from Lemma \ref{14} that the flowcone labelled by $C'$ is contained in $\overline{\mathrm{Min}(C)}$. \\

The map $\psi$ takes $\mathcal{D}(p)$ to a union of flowcones, with the choices described above.\\

In Theorem \ref{onlyfaces} it will be seen that for filling $C'$, any remaining choices in defining flowcones are not important. When $C'$ does not fill, for any choice of $C''$, $p$ is not a minimum of $C''$. This is because $-v_{C''}$ is in the tangent cone to $E(C')\subset E(C'')$ (this uses the assumption that the face labelled by $C'$ is not folded). Consequently Lemma 3.7 of \cite{MS} implies that any choice of $C''\subset C'$ for which the span of $\{-\nabla L(c)(p)\ |\ c\in C''\}$ is equal to the span of $\{-\nabla L(c)(p)\ |\ c\in C'\}$ satisfies $m(C')=m(C'')$. In the model of the thin part of $\mathrm{Min}(C)$ from Section \ref{horizonsec} one sees that the paths end in the same  minima independent of choices. \\

Suppose a vector in a face of $\mathcal{D}(p)$ labelled by $C'\subset C$ can be written as a positive linear combination of $\{-\nabla L(c)(p)\ |\ c\in C_{1}\subset C'\}$ and as a positive linear combination of $\{-\nabla L(c)(p)\ |\ c\in C_{2}\subset C'\}$ where $\{-\nabla L(c)(p)\ |\ c\in C_{1}\}$ and $\{-\nabla L(c)(p)\ |\ c\in C_{2}\}$ have the same span. When the face is not folded, $p$ is not a minimum of $C'$, because $-v_{C'}(p)\neq 0$ is in the tangent cone to $\mathrm{Sys}(C')$. It then follows from polytopal regularity that any point $x\in \overline{\mathrm{Min}(C)}$ at which $C_{1}$ is eutactic, $C_{1}\cup C_{2}$ is also eutactic because $\{-\nabla L(c)(x)\ |\ c\in C_{2}\}$ is in the span of $\{-\nabla L(c)(x)\ |\ c\in C_{1}\}$. Consequently, by Lemma \ref{14}, $\mathrm{Min}(C_{1})\subset \mathrm{Min}(C')$, and $\mathrm{Min}(C_{1})$ and $\mathrm{Min}(C')$ have the same dimension.\\

\begin{example}[The map $\psi$ is not always a homeomorphism on the boundary of $\mathcal{D}(p)$]
\label{nonhomeo}
When $p$ is the Bolza surface, the six curves shown in Figure \ref{genus2example1} label a stratum of dimension 1 adjacent to $p$. The gradients of these six curves are linearly independent along this stratum except at the critical point shown in Figure \ref{genus2example1}, where the dimension of the span of the gradients of the lengths drops to three. Any filling subset of this set of curves also determines a stratum adjacent to $p$ and this 1-dimensional stratum. The six curves determine the same equivalence class of strata as the four and five-dimensional strata adjacent to it. Although the flowcone is six dimensional, the ``base'' of the cone, i.e. the image of the facet of $\mathcal{D}(p)$ labelled by the six curves, is collapsed to a three dimensional cell.
\end{example}
\subsection{Proof of the theorem}
\label{proofthmonlyfaces}
Now that all the necessary concepts have been introduced, the purpose of this subsection is to prove the next theorem.

\begin{theorem}[Theorem \ref{onlyfaces} of the Introduction]
Suppose $C$ is a set of curves for which $\mathrm{Min}(C)$ has polytopal regularity and $p$ is a local maximum of $f_{\mathrm{sys}}$ with set of systoles $C$. Then the sets of minima making up the faces of $\mathrm{Min}(C)$ are in 1-1 correspondence with equivalence classes of strata adjacent to $p$ in $\mathcal{P}_{g}$.
\end{theorem}

Let $p\in \mathrm{Min}(C)$ be a local maximum of $f_{\mathrm{sys}}$. Filling subsets of $C$ that are not the labels of equivalence classes of strata adjacent to $p$ give sets of minima contained in the interior of $\mathrm{Min}(C)$. In most cases this is intuitively clear, as such sets of minima usually have petals passing through them. This is illustrated in Figure \ref{petalpassingthrough}.  \\

\begin{figure}
\centering
\includegraphics[width=6cm]{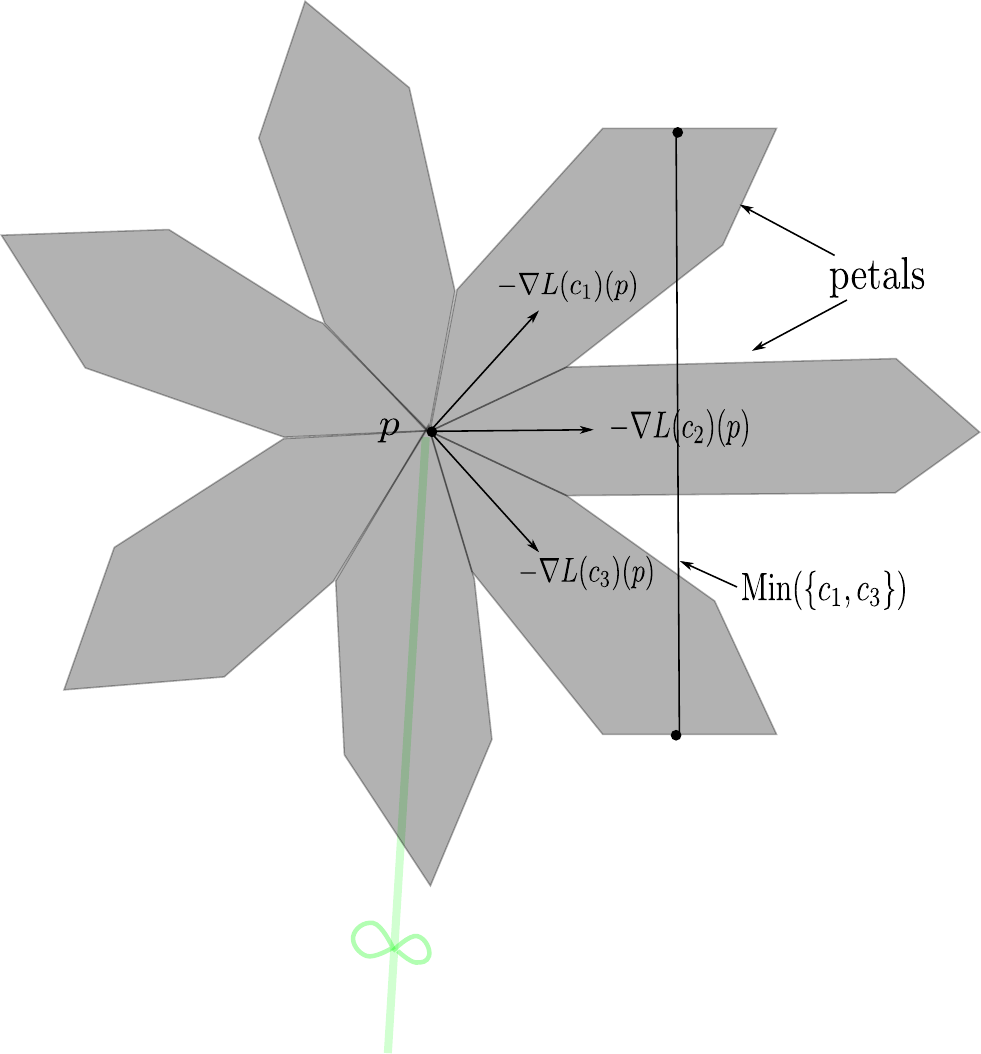}
\caption{Suppose $C$ contains $\{c_{1}, c_{2}, c_{3}\}$ whose negative gradients are shown. This figure shows a (lower dimensional representation of) a set of minima $\mathrm{Min}(\{c_{1}, c_{2}\})$ that does not determine a face of $\mathrm{Min}(C)$. It has the petal labelled by $c_{2}$ passing through it. }
\label{petalpassingthrough}
\end{figure}

\begin{proof}
Throughout this proof it will be assumed that faces of $\mathcal{D}(p)$ are not folded. This can be done without loss of generality, because if a face is folded, identical arguments apply with the faces replaced by the folded faces.\\

In examples such as Example \ref{nonhomeo}, adjacent strata in the same equivalence class will be treated as a single stratum with dimension given by that of the largest dimensional representative of the equivalence class. In Example \ref{nonhomeo}, this is a stratum of dimension three.\\

Choose a facet of $\mathcal{D}(p)$ with vertices $\{\nabla L(c)(p)\ |\ c\in C'\}$ and construct its flowcone. Assume that the stratum labelled by $C'$ is not in the same equivalence class as any higher dimensional strata adjacent to it at $p$. A stratum with this property must exist, because by Proposition 17 of \cite{SchmutzMorse} the union of the flowcones corresponding to the folded faces of $\mathcal{D}(p)$ is not all of $\mathcal{T}_{g}$, and hence the union has a codimension 1 boundary. These assumptions imply that $\mathrm{Min}(C')$ has codimension 1 in $\mathcal{T}_{g}$ and that by Lemma \ref{differentminima} and polytopal regularity, the set of points in $\mathrm{Min}(C)$ that are generalised minima of $C'$ is an embedded codimension 1 submanifold.\\

The first step is to prove the following claim.\\

Claim: the nonnegative linear combinations of vectors corresponding to the boundary of the facet of $\mathcal{D}(p)$ labelled by $C'$ are tangent to a set of paths emanating from $p$ in the flowcone labelled by $C'$. This set of paths are separating in $\mathrm{Min}(C)$. The sets of minima $\mathrm{Min}(C_{i})$ for $C_{i}\subset C'$ are all contained in or to the same side of these paths, whereas sets of minima $\mathrm{Min}(C_{j})$ for $C_{j}\subset C$ but $C_{j}\not\subset C'$ are either on the other side of the union of paths or intersect both sides.\\

Suppose $c''\in C\setminus C'$. By assumption, $-\nabla L(c'')(p)$ is not in the convex hull of $\{\nabla L(c)(p)\ |\ c\in C'\}$. If it can be shown that this holds for every $c''\in C\setminus C'$ at every point of $\mathrm{Min}(C)$, then $\mathrm{Min}(\{c''\})^{\infty}$ is on the opposite side of the paths to $\mathrm{Min}(C')$. The claim then follows because by Corollary \ref{defncor1}, $\mathrm{Min}(\{c''\})^{\infty}$ is adherent to $\mathrm{Min}(C_{j})$ whenever $C_{j}$ contains $\{c''\}$. \\

Polytopal regularity prevents $-\nabla L(c'')(x)$ from being in $\{-\nabla L(c)(x)\ |\ c\in C'\}$ at $x\in \mathrm{Min}(C)$ with one exception. For $-\nabla L(c'')$ to enter the convex hull of $\{-\nabla L(c)\ |\ c\in C'\}$ from outside, it must pass though the closure of a facet of the convex hull of $\{\nabla L(c)\ |\ c\in C'\}$. Denote (one such) facet by $f$, and the set of curves labelling it by $C(f)$. When $-\nabla L(c'')$ passes through the convex hull of $\{-\nabla L(c)\ |\ c\in C(f)\}$, the dimension of the span of $\{-\nabla L(c)\ |\ c\in C(f)\cup\{c''\}\}$ drops. This can only happen at a generalised minimum of $C(f)\cup\{c''\}$. This proves the claim in the case that the generalised minima of $C(f)\cup\{c''\}$ are not a separating set in $\mathrm{Min}(C)$, with $p$ on one side and points of $\mathrm{Min}(C')$ on the other side.\\

Figure \ref{thepolytope} illustrates a configuration that might be used to try to find a contradiction to the claim. It will be shown that if $-\nabla L(c'')$ is able to enter the convex hull of $\{-\nabla L(c)\ |\ c\in C'\}$, $C'$ could not have been the label of a facet of $\mathcal{D}(p)$. \\

In Figure \ref{thepolytope}, note the symmetry between $C(f)\cup \{c''\}$ and $C'$. If the claim breaks down for $C'$, it must also break down with $C(f)\cup \{c''\}$ in place of $C'$. This is because $p$ is a minimum of $C'\cup\{c''\}$, and $\{\nabla L(c)(p)\ |\ c\in C'\cup\{c''\}\}$ span $T_{p}\mathcal{T}_{g}$. Consequently, polytopal regularity implies $\mathrm{Min}(C'\cup\{c''\})$ has dimension $6g-6$. Also, if the claim breaks down, $\mathrm{Min}(C(f))$ is in the interior of $\mathrm{Min}(C)$, and hence by Lemma \ref{14} so are the faces $\mathrm{Min}(C')$ and $\mathrm{Min}(C(f)\cup\{c''\})$ of $\mathrm{Min}(C'\cup\{c''\})$ adjacent to $\mathrm{Min}(C(f))$ in $\mathrm{Min}(C'\cup\{c''\})$. If the claim breaks down and $f$ is the facet through which $-L(c'')$ passes, then $\mathrm{Min}(C(f))$ is in the interior of $\mathrm{Min}(C)$. This will therefore be assumed without loss of generality.\\

The symmetry between $C(f)\cup \{c''\}$ and $C'$ from Figure \ref{thepolytope} implies that if the claim breaks down, the point $p$ is contained in a region, called a ``wedge'', of $\mathrm{Min}(C)$. On one side this wedge is bounded by generalised minima of $C(f)\cup \{c''\}$ and on another side it is bounded by generalised minima of $C'$. The wedge diametrically opposite the one containing $p$ contains $\mathrm{Min}(C'\cup\{c''\})$.\\

\begin{figure}
\centering
\includegraphics[width=8cm]{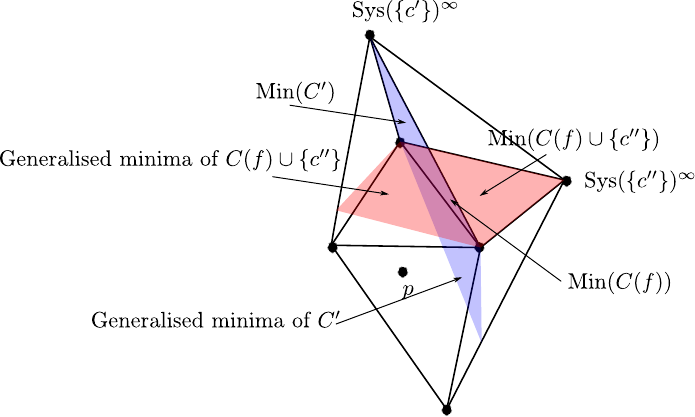}
\caption{An attempt at finding a contradiction to the claim. The generalised minima of $C'$ are shaded in blue, and the generalised minima of $C(f)\cup\{c''\}$ in red. The generalised minima of $C(f)$ are in the intersection of these. }
\label{thepolytope}
\end{figure}

It can and will be assumed without loss of generality that $\{-\nabla L(c)(p)\ |\ c\in C(f)\}$ is linearly independent, because otherwise $C(f)$ could be replaced by a subset whose gradients are linearly independent with the same span, and such that $-\nabla L(c'')$ passes through the convex hull of the gradients of this subset. Similarly it will also be assumed without loss of generality that $C'\setminus C(f)$ consists of a single curve $c'$ for which $\nabla L(c')(p)$ is linearly independent from $\{\nabla L(c)(p)\ |\ c\in C(f)\}$. If $C'\setminus C(f)$ consists of more than one curve, these do not affect the argument and merely complicate the notation. \\

Recall that by assumption $\mathrm{Sys}(C')$ is 1-dimensional, $\mathrm{Min}(C')$ has codimension 1 and $C(f)$ is a set of curves labelling a facet of the convex hull of $\{\nabla L(c)(p)\ |\ c\in C'\}$. It follows that $\mathrm{Min}(C(f))$ has codimension 2 in $\mathcal{T}_{g}$, and the dimension of the span of $\{\nabla L(c)\ |\ c\in C(f)\}$ only drops by one at the minima of $C(f)$. As in the proof of Lemma \ref{differentminima}, multiple applications of the pre-image theorem implies that every connected component of $E(C(f))\cap \mathrm{Min}(C)$ is a 2-dimensional embedded submanifold. Moreover, since $\mathrm{Min}(C(f))$ is in the interior of $\mathrm{Min}(C)$, this also gives that each connected component of $E(C(f))\cap\mathrm{Min}(C)$ has a unique generalised minimum $q$ of $C(f)$ for which $L(c_{f})|_{E(C(f))}$ is strictly convex for some (and hence every) curve $c_{f}\in C(f)$. This is shown in Figure \ref{rateofdecrease}.\\

\begin{figure}
\centering
\includegraphics[width=6cm]{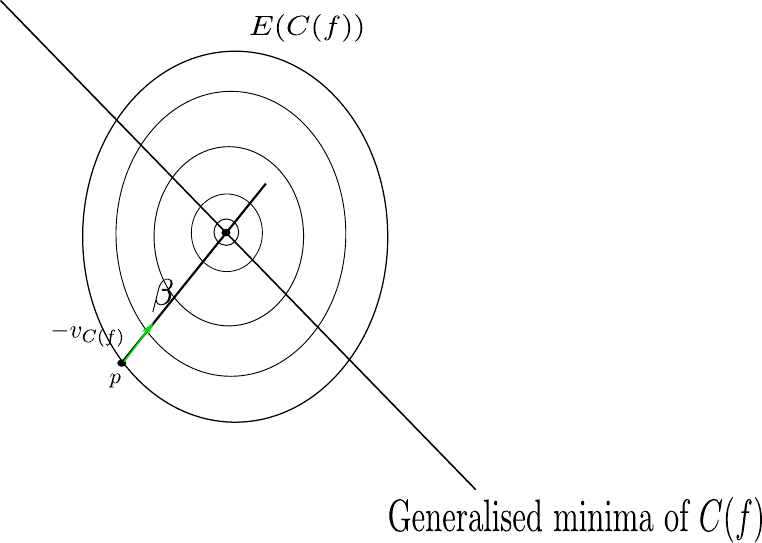}
\caption{The locus $E(C(f))$. The circles are level sets of $L(c_{f})|_{E(C(f))}$. The generalised minima of $C(f)$ intersect $E(C(f))$ only in the point $q$, and this intersection is transverse.}
\label{rateofdecrease}
\end{figure}


Within $\mathrm{Min}(C)$, the loci $E(C(f)\cup\{c'\}, d)$ are codimension 1 embedded submanifolds of $E(C(f))$ that foliate $E(C(f))$ for varying $d$. Similarly for $E(C(f)\cup\{c''\}, d)\subset E(C(f))$. There is a smooth path $\beta$ in $\mathrm{Min}(C)$ with $\beta(0)=p$ and $\dot{\beta}=-v_{C(f)}$ passing through the unique minimum $q$ of $C(f)$ on the connected component of $E(C(f))$. The segment of $\beta$ between $p$ and $q$ is contained in the same wedge as $p$ because (assuming the metric $\mathbf{g}$) the vector field $-v_{C(f)}$ is contained in the tangent space to the manifolds of generalised minima of $C(f)\cup\{c'\}$ and of $C(f)\cup\{c''\}$ on the boundary of the wedge.\\

Still assuming the metric $\mathbf{g}$, $-\nabla L(c')(q)$ is in the tangent space to the submanifold of generalised minima of $C(f)\cup\{c'\}$, with a nonvanishing component into the boundary of the wedge diametrically opposite the one containing $p$. Similarly for $-\nabla L(c'')(q)$. The same holds in the intersection with $E(C(f))$, because $E(C(f))$ intersects the submanifold of generalised minima of $C(f)$ at right angles with respect to the metric $\mathbf{g}$. The projection of $-\nabla L(c')(q)$ to $T_{q}E(C(f))$ is a direction in which $L(c')$ is decreasing and the lengths of curves in $C(f)$ are stationary. Similarly for $-\nabla L(c'')(q)$. It follows that there is a cone of directions in $T_{q}E(C(f))$ in which both $L(c')$ and $L(c'')$ are decreasing. This cone is shown in Figure \ref{orangeconea}, and will be called the \textit{cone of decrease}.\\

\begin{figure}
\centering
\includegraphics[width=7cm]{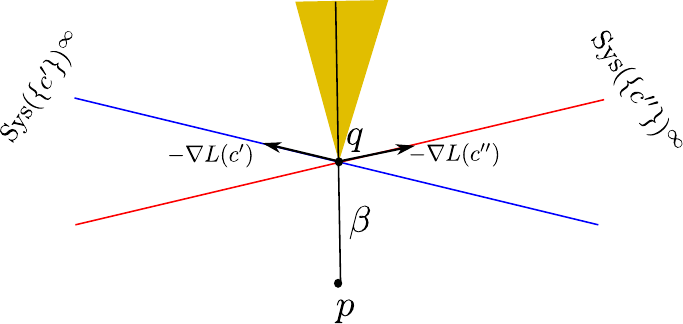}
\caption{The cone of decrease is shown in orange. In this figure, the tangent vector to $\beta$ at $q$ is in the cone of decrease. The red and blue lines are the sets of generalised minima that define the wedges.}
\label{orangeconea}
\end{figure}

Consider first the case shown in Figure \ref{orangeconea} in which the tangent to $\beta$ at $q$ is in the interior of the cone of decrease. Choose a $c_{f}\subset C_{f}$; it does not matter which. The function $d_{1}:=L(c')-L(c_{f})$ is strictly monotone on $\beta$ between $p$ and $q$, because otherwise there would be an $E(C(f)\cup\{c'\}, d^{*})$ intersecting $\beta$ on which $d_{1}|_{E(C(f))}$ is stationary or $v_{C(f)\cup\{c'\}}=v_{C(f)}$. By Remark \ref{balancedeverywhere}, the second option could only happen when $E(C(f)\cup\{c'\}, d^{*})$ is semi-balanced and contains $\beta$, which could only occur when the tangent to $\beta$ at $q$ is on the boundary of the cone of decrease. The first possibility contradicts polytopal regularity, because by polytopal regularity, linear independence of $\nabla L(c')$ from $\{\nabla L(c)\ |\ c\in C(f)\}$ holds everywhere on the 1-dimensional $E(C(f)\cup\{c'\}, d^{*})\cap \mathrm{Min}(C)$ except the point at which $E(C(f)\cup\{c'\}, d^{*})$ intersects the manifold of generalised minima of $C(f)\cup\{c'\}$. This shows $L(c')$ is decreasing faster along $\beta$ than $L(c_{f})$ at every point of the connected component of $\beta\setminus \{q\}$, and hence also at $p$.\\

Since $L(c')$ is increasing faster than $L(c_{f})$ in the direction of $v_{C(f)}(p)$, it follows that $v_{C(f)\cup\{c'\}}(p)$ is on the opposite side of the subspace of $T_{p}\mathcal{T}_{g}$ spanned by $\{\nabla L(c)(p)\ |\ c\in C(f)\}$ as $\nabla L(c')(p)$. By a symmetric argument, the same holds for $v_{C(f)\cup\{c''\}}(p)$ and the subspace of $T_{p}\mathcal{T}_{g}$ spanned by $\{\nabla L(c)(p)\ |\ c\in C(f)\}$. However, since $\nabla L(c')(p)$ and $\nabla L(c'')(p)$ are both on opposite sides of the subspace of $T_{p}\mathcal{T}_{g}$ spanned by $\{\nabla L(c)(p)\ |\ c\in C(f)\}$, a contradiction is obtained to the assumption that $C'$ is the label of a facet of $\mathcal{D}(p)$. This concludes the proof of the claim when the tangent to $\beta$ at $q$ is in the cone of decrease.\\

\begin{figure}
\centering
\includegraphics[width=5cm]{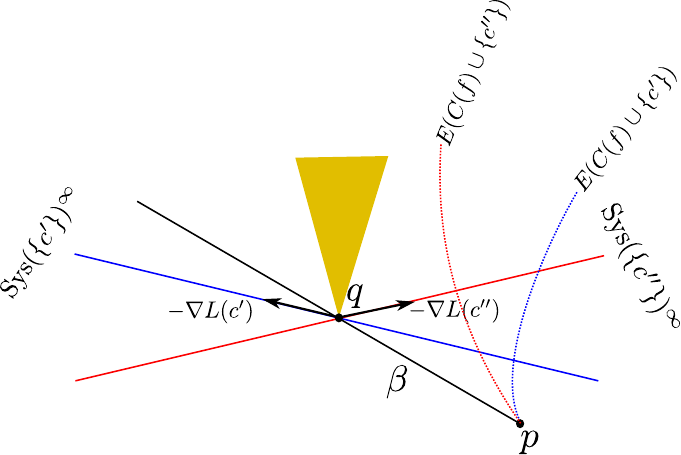}
\caption{}
\label{orangeconeb}
\end{figure}

A similar contradiction is obtained when the tangent to $\beta$ at $q$ is in the boundary of the cone of decrease. In this case, one of $E(C(f)\cup\{c'\})$ and $E(C(f)\cup\{c''\})$ is semi-balanced, and the other is balanced.\\

If the tangent to $\beta$ at $q$ is not in the cone of decrease, it gives a direction in which one of $L(c')$ or $L(c'')$ is increasing at $q$ and the other is decreasing at $q$. Suppose $L(c'')$ is increasing. As in the previous case, the function $d_{2}:=L(c'')-L(c_{f})$ is strictly monotone on $\beta$ between $p$ and $q$. Since $d_{2}$ is increasing at $q$ it is increasing at $p$. Similarly, $d_{1}$ is monotone decreasing along $\beta$ between $p$ and $q$. This implies the tangent vector $-v_{C(f)\cup\{c'\}}(p)$ to $E(C(f)\cup\{c'\})$ and the tangent vector $-v_{C(f)\cup\{c''\}}(p)$ to $E(C(f)\cup\{c''\})$ are both to the same side of the complement in $T_{p}E(C(f))$ of the subspace spanned by $v_{C(f)}$, as shown in Figure \ref{orangeconeb}.\\

Since $\mathrm{Sys}(C')$ is a 1-dimensional stratum, up to rescaling, $-v_{C(f)\cup\{c'\}}(p)$ is unique, and must be in the tangent cone to $\mathrm{Sys}(C')$ at $p$. However, $-v_{C(f)\cup\{c'\}}(p)$ is a direction in which the length of $L(c'')$ is decreasing faster than the lengths of curves in $C'$. This contradiction completes the proof of the claim.\\

It follows from the claim that $\mathrm{Min}(C')$ is separating in $\mathrm{Min}(C)$. It also follows from the claim that $\mathrm{Min}(C')$ cannot have larger dimensional sets of minima on the side away from $p$ because any such larger dimensional sets of minima would have to be labelled by subsets of $C$ strictly containing $C'$, which the claim showed is impossible.\\

Either
\begin{enumerate}
\item{$\mathrm{Min}(C')$ is a face of $\mathrm{Min}(C)$,  }
\item{$\mathrm{Min}(C')$ is contained in a larger face of $\mathrm{Min}(C)$ of the same dimension, or}
\item{$\mathrm{Min}(C')$ has dimension $6g-6$ and is therefore not a face.}
\end{enumerate}

Point (2) is not possible, because for any curve $c\in C\setminus C'$ for which $-\nabla L(c)(p)$ is not a vertex on the boundary of the facet of $\mathcal{D}(p)$ labelled by $C'$, $p$ is a minimum of $C'\cup \{c\}$. Loci $E(C'\cup \{c\}, d)$ in $\mathrm{Min}(C)$ are 0-dimensional, but by polytopal regularity would have to be at least 1-dimensional if $\mathrm{Min}(C'\cup \{c\})$ were a face. Point (3) is similarly ruled out by the fact that loci $E(C',d)$ in $\mathrm{Min}(C)$ are 1-dimensional (same as $\mathrm{Sys}(C')$) and hence polytopal regularity does not allow $\mathrm{Min}(C')$ to have dimension $6g-6$. This concludes the proof that 1-dimensional strata adjacent to $p$, that are not in the same equivalence class as higher dimensional strata, determine facets of $\mathrm{Min}(C)$. \\

The proof that higher dimensional strata adjacent to $p$ determine faces of $\mathrm{Min}(C)$ is similar. The argument is simpler, as the manifolds of generalised minima do not have high enough dimension to be separating in $\mathrm{Min}(C)$, so it is not necessary to consider the phenomenon illustrated in Figure \ref{thepolytope}.  Note that by the duality between $\mathcal{D}(p)$ and $\mathcal{V}(p)$, the folded faces of the facet of $\mathcal{D}(p)$ labelled by $C'$ correspond to strata adjacent to $\mathrm{Sys}(C')$ at $p$. This completes the proof that every equivalence classe of strata adjacent to $p$ in $\mathcal{P}_{g}$ determines a face of $\mathrm{Min}(C)$.\\


The converse follows from observing that the union of the flowcones corresponding to the facets of $\mathcal{D}(p)$ give a cell of dimension $6g-6$. The base of a flowcone labelled by $C''\subset C$ is either $\mathrm{Min}(C'')$ if $C''$ fills, or $\mathrm{Min}(C'')^{\infty}$ if $C''$ does not fill. There is no space to add in any extra faces that do not come from equivalence classes of strata. 
\end{proof}

\begin{remark}
\label{adjacencyinfo}
The claim in the proof of Theorem \ref{onlyfaces} ensures that the adjacency information between the petals, encoded in $\mathcal{D}(p)$, is preserved throughout $\mathrm{Min}(C)$ and agrees with the combinatorial information given by the horizon map.
\end{remark}

\begin{remark}
\label{onlyfacesrem}
For a critical point $p$ that is not a local maximum, Theorem \ref{onlyfaces} also applies to the strata adjacent to $p$ and intersecting $\mathrm{Min}(C)$, with identical proof.
\end{remark}



\section{Defining duality}
\label{dualitysec}
The purpose of this section is to explain what a dual to $\mathcal{P}_{g}$ should be, and to outline how to construct duals from sets of minima. As $\mathcal{P}_{g}$ is not an embedded submanifold, the definitions used here are not quite the standard definitions. In the process it will be made rigorous the sense in which Schmutz Schaller's sets of minima are dual to Thurston's $\mathcal{P}_{g}$. Theorem \ref{onlyfaces} is part of this story, but it remains to understand how the sets of minima intersect $\mathcal{P}_{g}$ and how duals are constructed from sets of minima. This makes it possible to label duals by sets of curves; an observation that will be exploited in a later paper to relate the homological span of the sets of curves to collapsibility properties of $\mathcal{P}_{g}$.

\subsection{Intersections}
\label{facesandintersections}
It was shown in Section 5 of \cite{MS} that near a critical point $p\in \mathrm{Sys}(C)$, the set of minima $\mathrm{Min}(C)$ behaves like the stable manifold of $p$. As a result of Corollary \ref{defncor1}, the same is true in the thin part of $\mathrm{Min}(C)$. This subsection shows that certain sets of minima have the right intersection properties with strata of $\mathcal{P}_{g}$ to define duals. \\

The model of $\mathcal{P}_{g}$ around a critical point $p$ from \cite{MS} shows that when $T_{p}\mathrm{Min}(C)$ does not have any vectors in the tangent cone to $\mathcal{P}_{g}$, $\mathcal{P}_{g}$ has a tangent space at $p$, and this is given by $\{\nabla L(c)(p)\ |\ c\in C\}^{\perp}$. In this case, the dimension of $\mathrm{Min}(C)$ at $p$ is equal to the codimension of $\mathcal{P}_{g}$ at $p$.

\begin{theorem}[Theorem \ref{neededforduality} of the Introduction]
Suppose $p$ is a critical point with set of systoles $C$ and for which $T_{p}\mathrm{Min}(C)$ does not have any vectors in the tangent cone to $\mathcal{P}_{g}$. Then $\mathrm{Min}(C)$ is a cell with empty boundary. In addition, there is a homotopy of $\mathrm{Min}(C)$ to a cell obtained as the pre-image of $p$ under Thurston's equivariant deformation retraction of $\mathcal{T}_{g}$ onto $\mathcal{P}_{g}$. This homotopy fixes $p$ and keeps the thin part of $\mathrm{Min}(C)$ in the thin part of $\mathcal{T}_{g}$.
\end{theorem}

\begin{remark}
Under the assumptions of the theorem, a cell obtained as the pre-image of $p$ under Thurston's deformation retraction intersects $\mathcal{P}_{g}$ in the single point $p$. As explained below, a stronger statement appears to be true; one could choose a vector field in  Thurston's construction similar to the choice made by Thurston such that $\mathrm{Min}(C)$ is the pre-image of $p$ under the resulting equivariant deformation retraction.
\end{remark}
\begin{proof}
Suppose $C'$ is the label of a facet of $\mathcal{D}(p)$. Let $\alpha(t)$ be a path in $\mathrm{Min}(C)$ with $\alpha(0)=p$ and $\dot{\alpha}(t)$ a positive linear combination of $\{-\nabla L(c)\ |\ c\in C'\}$. Gradients are defined with respect to $\mathbf{g}$. Note that, unlike in Subsection \ref{twocases}, there are now no regularity assumptions being made.\\

Claim - The assumption that $T_{p}\mathrm{Min}(C)$ does not have any vectors in the tangent cone to $\mathcal{P}_{g}$ implies that $C'$ does not fill. \\

To prove the claim, suppose $C'$ is the label of a face of $\mathcal{D}(p)$ that is not folded. If $C'$ were a filling set,
this gives a stratum adjacent to $p$ with filling systoles and tangent cone at $p$ containing a vector $-v_{C'}$ in the cell of $\mathcal{V}(p)$ labelled by $C'$. This contradicts the assumption in the statement of the theorem. Similarly if the facet is folded, the argument given in the proof of Theorem 4.5 of \cite{MS} shows that the folded faces are also labelled by filling sets, each of which has a tangent cone at $p$ that contradicts the assumption in the statement of the theorem. It follows that for sufficiently small $t$, the set of systoles along $\alpha$ is a nonfilling subset $C''$ of $C'$. This concludes the proof of the claim.\\

Since $C'$ does not fill, $\alpha$ extends into the thin part of $\mathcal{T}_{g}$. Assuming the metric $\mathbf{g}$, by Lemma \ref{14}, $\alpha$ is contained in $\mathrm{Min}(C)$. Consequently, $\mathrm{Min}(C)$ has no boundary. Since $\mathrm{Min}(C)$ has no boundary, it follows from Theorem 10 of \cite{SchmutzMorse} that $\mathrm{Min}(C)$ is a continuously differentiable cell.\\

Define $\mathcal{P}_{g,\epsilon'}$ to be the set of all $x\in \mathcal{T}_{g}$ such that the set $C_{\mathrm{short}}(x)$ of curves with length at $x$ less than or equal to $f_{\mathrm{sys}}(x)+\epsilon'$ fill.\\

Let $\mathcal{N}(p)$ be a ball contained in $\mathrm{Min}(C)$ centered on $p$ and intersecting every choice of path $\alpha$ in an interval on which the systoles are contained in a nonfilling subset of $C$. This is possible by local finiteness. Choose $\epsilon>0$ needed in the definition of the vector field used in constructing Thurson's equivariant deformation retraction of $\mathcal{T}_{g}$ onto $\mathcal{P}_{g}$. Here $\mathcal{N}(p)$ is chosen sufficiently small such that the boundary of $\partial\mathcal{N}(p)$ is contained in $\mathcal{P}_{g,\epsilon'}$.\\

Define $\mathcal{B}(p)$ to be the union of $\mathcal{N}(p)$ and the pre-image of $\partial\mathcal{N}(p)$ under Thurston's deformation retraction. By construction, the cell $\mathcal{B}(p)$ has the same dimension as $\mathrm{Min}(C)$.\\

A path $\alpha$ in $\mathrm{Min}(C)$ will be made to correspond to a path $\tilde{\alpha}$ in $\mathcal{B}(p)$, with $\tilde{\alpha}(0)=\alpha(0)=p$, and $\dot{\tilde{\alpha}}(0)=\dot{\alpha}(0)$. The path $\tilde{\alpha}$ is a flowline of the $f_{\mathrm{sys}}$-increasing flow used in \cite{Thurston} to construct the deformation retraction. Some information about Thurston's construction will now be needed to argue that $\alpha$ stays close to $\tilde{\alpha}$. More details can be found in \cite{ThurstonSpineSurvey}.\\

An open cover of $\mathcal{T}_{g}$ is obtained as follows: A finite set of curves $C_{i}$ determines an open set $\mathcal{U}_{i}$ in the cover, where $\mathcal{U}_{i}$ is given by the set of points of $\mathcal{T}_{g}$ for which $C_{i}$ is the set of curves of length less than $f_{\mathrm{sys}}+|C_{i}|\epsilon$. By compactness of the $\delta$-thick part of $\mathcal{T}_{g}$ modulo the action of $\Gamma_{g}$, $\epsilon$ can and will be chosen small enough that every set $C_{i}$ corresponding to a nonempty open set $\mathcal{U}_{i}$ has the property that curves in $C_{i}$ intersect pairwise at most once.\\

A smooth vector field $X_{C_{i}}$ with support on $\mathcal{U}_{i}$ is then constructed. When $C_{i}$ does not fill, $X_{C_{i}}$ is nonvanishing on $\mathcal{U}_{i}$ and gives a direction in which the lengths of all the curves in $C_{i}$ are strictly increasing. Partitions of unity are used to glue together the vector fields with local support to obtain a smooth vector field $X$ on $\mathcal{T}_{g}$ that generates the flow. \\

Section 6 of \cite{MS} discusses possible choices for the vector fields $\{X_{C_{i}}\}$. All that will be needed here is that $X_{C_{i}}$ is in the convex hull of $\{\nabla L(c)\ |\ c\in C_{i}\}$. When $x_{1}:=\tilde{\alpha}(t_{1})$ is in the intersection of the open sets $\mathcal{U}_{1}, \ldots, \mathcal{U}_{k}$, it follows that $X(x_{1})$ is in the convex hull of $\{\nabla L(c)(x_{1})\ |\ c\in C_{1}\cup\ldots\cup C_{k}\}$. \\

If $c'$ is a curve that is not contained in the subsurface of $\mathcal{S}_{g}$ filled by $C_{i}$, for $c'$ to enter a set of shortest curves $C_{i}$ at a point $x_{2}:=\tilde{\alpha}(t_{2})$, $t_{2}>t_{1}$, for sufficiently small $\epsilon$, $L(c')$ must be decreasing along $\tilde{\alpha}$ at $x_{1}$ faster than the lengths of curves in $C_{i}$ for $x_{1}$ sufficiently close to $x_{2}$.\\

Claim: $\tilde{\alpha}$ only passes through open sets labelled by sets of curves contained in the subsurface of $\mathcal{S}_{g}$ filled by $C'$.\\

Local finiteness implies that for small $t_{2}$, $C_{1}\cup\ldots\cup C_{k}\subset C'$. The claim is proven by finding a contradiction to the statement that for $c'$ not contained in the subsurface filled by $C'$, $L(c')$ is decreasing along $\tilde{\alpha}$ at $x_{1}$ faster than the lengths of curves in $C_{1}\cup\ldots\cup C_{k}$. This proves the claim at the point $x_{2}$. The same argument can then be applied along $\alpha$ at larger and larger values of $t$ to prove the claim.\\
 
Outside of a set of measure 0, $-\nabla L(c')$ is not in or on the boundary of the convex hull of $\{\nabla L(c)\ |\ c\in C'\}$. The simplest way to see this is to note that by assumption, there exists a curve $c'*$ intersecting $c'$ and disjoint from every curve in $C'$. By Wolpert's twist formula, away from a set of measure 0, changing the twist parameter around $c'*$ gives a direction in which $L(c')$ is changing but the lengths of curves in $C'$ are not. When the loci $E(C'\cup\{c'\},d)$ have dimension greater than zero, a stronger statement follows from Lemma 3.3 of \cite{MS}.\\

It is possible to assume without loss of generality that at $x_{1}$, $-\nabla L(c')(x_{1})$ is not in or on the boundary of the convex hull of $\{\nabla L(c)(x_{1})\ |\ c\in C'\}$. Otherwise, if $c'$ were to become short enough to be contained in a set  labelling an element of the cover, this holds on an open set, so it would be possible to choose a different path in place of $\alpha$ or a different point in place of $x_{1}$ to derive the contradiction.\\

Consider the face $f$ of the convex hull $\{\nabla L(c)(x_{1})\ |\ c\in C'\}$ closest to $-\nabla L(c')(x)$, and denote by $C(f)$ the subset of $C'$ corresponding to $f$. By assumption, $-\nabla L(c')(x_{1})$ is linearly independent from $\{\nabla L(c)(x_{1})\ |\ c\in C(f)\}$.\\

For every subset $C'(f)$ of $C(f)$ for which $\{\nabla L(c)(x_{1})\ |\ c\in C'(f)\}$ are linearly independent, when $x_{1}$ is sufficiently close to $x_{2}$, for $L(c')$ to have length almost equal shortest, Proposition \ref{29} implies that in the set $\mathcal{US}(C'(f)\cup\{c\})$ of unit vectors in the span of $\{\nabla L(c)(x_{1})\ |\ c\in C'(f)\cup\{c'\}\}$, the vector $v_{C'(f)\cup\{c'\}}$ is in the same connected component of $\mathcal{US}(C'(f)\cup\{c\})\setminus \mathcal{US}(C'(f))$ as $\nabla L(c')$. Consequently, the vectors in the convex hull of $\{\-nabla L(c)(x_{1})\ |\ c\in C(f)\}$ are directions in which $L(c')$ is decreasing more slowly than the length of some curve in $C(f)$. The same therefore holds for the convex hull of $\{\nabla L(c)(x_{1})\ |\ c\in C'\}$. This contradiction completes the proof of the claim.\\

Since the systoles along $\tilde{\alpha}$ are all contained in the subsurface of $\mathcal{S}_{g}$ filled by $C'$, the endpoint of $\tilde{\alpha}$ is in a stratum $\mathrm{Sys}(m'(C'))$, where $m'(C')$ is a multicurve disjoint from $m(C')$. It follows that $\alpha$ and $\tilde{\alpha}$ have endpoints in either the same or adjacent strata. This is important, because it will mean that the homotopy does not move the points in the thin part very far.\\

The cell $\mathcal{B}(p)$ is parameterised as follows: Choose the parameterisation of the paths $\alpha$ and $\tilde{\alpha}$ such that $t$ varies over $[0,1]$, and varies continuously over nearby paths. Set $C=\{c_{1}, \ldots, c_{n}\}$, and make a choice of parameters $(a_{1}, \ldots, a_{n})$ varying continously over $\mathrm{Min}(C)$, where $x\in \mathrm{Min}(C)$ is the point at which the minimum of the length function $\sum_{i=1}^{n}a_{i}(x)L(c_{i})$ is realised. A point $\tilde{\alpha}(t)$ in $\mathcal{B}(p)$ is assigned the same value of the tuple $(a_{1}, \ldots, a_{n})$ as the point $\alpha(t)$ in $\mathrm{Min}(C)$. This determines a smooth vector field on $\mathcal{B}(p)$, with value at $x\in \mathcal{B}(p)$ given by $\nabla L(A(x), C)$. \\

A map $H'$ is obtained, that takes $(x,t)$ for $x$ in $\mathcal{B}(p)$ and $t\geq 0$ to the point $\eta(t)$, where $\eta$ is a flowline of the vector field $-\nabla L(A(x), C)$ with $\alpha(0)=x$. If necessary, the vector field on $\mathcal{B}(p)$ is scaled by a smooth function such that $H'$ takes $\mathcal{B}(p)$ to $\mathrm{Min}(C)$ in unit time. The required homotopy $H$ is obtained by smoothly stretching/compressing the image of $\mathcal{B}(p)$ under $H'_{t}$ for $t\in[0,1]$ so that the homotopy preserves the subset of the cell that maps into the thin part of $\mathcal{T}_{g}$.
\end{proof}

\subsection{Constructing duals}
\label{constructingduals}
This subsection explains how to construct duals to the Thurston spine using sets of minima; previously this had been done using pre-images of points under Thurston's deformation retraction. The advantage of this approach is that, with the help of the horizon map, it is possible to label a dual by a set of curves. One reason for being interested in duals is that they can be used to relate the topology of the boundary of the thick part of $\mathcal{T}_{g}$ -- an object that will be identified with Harvey's curve complex by Theorem \ref{complexembedding} -- with $\Gamma_{g}$-equivariant collapsibility properties of $\mathcal{P}_{g}$.\\

Recall from Theorem \ref{MSthm} that $\mathcal{P}_{g}$ equivariantly deformation retracts onto a subcomplex $\mathcal{P}_{g}^{X}$ consisting of a union of unstable manifolds of critical points of $f_{\mathrm{sys}}$. The cell complex $\mathcal{P}_{g}^{X}$ is obtained from $\mathcal{P}_{g}$ via a deformation retraction that arises from the flow of an $f_{\mathrm{sys}}$-increasing $\Gamma_{g}$-equivariant vector field $X$ on $\mathcal{P}_{g}$, analogous to the construction of $\mathcal{P}_{g}$ from $\mathcal{T}_{g}$.\\

The fact that the existence and uniqueness theorem of ODEs is only guaranteed to hold where $f_{\mathrm{sys}}$ is smooth, unstable manifolds of distinct critical points are not necessarily disjoint; flowlines can flow together, but it was shown in Section 3.2 of \cite{JustVCD} that they cannot flow apart. In \cite{JustVCD}, connected components of intersections of unstable manifolds of distinct critical points were called flaps. Flaps cannot contain critical points. By collapsing the segments of flowlines contained in flaps, $\mathcal{P}^{X}_{g}$ can be shown to be $\Gamma_{g}$-equivariantly homotopy equivalent to a cell complex without flaps. As there is no need to construct duals to flaps, the possibility of flaps will be ignored for the remainder of this exposition.\\

For topological Morse functions, the definition of stable and unstable manifolds of critical points also requires some discussion. This is given in \cite{JustVCD} and \cite{Morse}. It will be assumed here that these have been defined to have analogous properties to the stable and unstable manifolds in the smooth case. In particular, the unstable manifold of a critical point $p$ is a cell in $\mathcal{P}_{g}^{X}$, with boundary consisting of a union of unstable manifolds of critical points of larger index. Details of pathological cases that can arise are given in \cite{JustVCD}.\\
 
\textbf{What is a dual?} To begin with, let $p\in \mathrm{Sys}(C)$ be a critical point satisfying the assumptions of Theorem \ref{neededforduality}. It was shown in \cite{MS} that $\mathcal{P}_{g}$ has a tangent space at $p$, and under the assumptions of Theorem \ref{neededforduality} the index of $p$ is equal to the codimension of this tangent space. This definition is $\Gamma_{g}$-equivariant, and by Theorems \ref{Akroutthm} and \ref{MSthm}, the dimension of $\mathrm{Min}(C)$ is equal to the codimension of $\mathcal{P}_{g}$ at $p$. It follows that, by Theorem \ref{neededforduality}, $\mathrm{Min}(C)$ has all the properties one would expect of a dual intersecting $\mathcal{P}_{g}$ in the point $p$.\\

Let $p$ be a critical point of $f_{\mathrm{sys}}$ in $\mathcal{P}_{g}^{X}$ for which the only vectors in $\mathcal{V}(p)\cap T_{p}\mathrm{Min}(C)$ in cells labelled by filling sets of curves are in the tangent cone to $\mathcal{P}_{g}\setminus \mathcal{P}_{g}^{X}$. In this case, $\mathrm{Min}(C)$ will also be called a dual to $\mathcal{P}_{g}^{X}$, even though $\mathrm{Min}(C)$ can have nonempty boundary. Since the sets of minima representing the boundary of $\mathrm{Min}(C)$ are disjoint from $\mathcal{P}_{g}^{X}$, the corresponding cells are homotopic to cells in the thin part of $\mathcal{T}_{g}$. Such homotopies are constructed by reversing the flow in the construction of the deformation retraction of $\mathcal{T}_{g}$ onto $\mathcal{P}_{g}^{X}$. Such homotopies leave the thin part of a cell invariant and restrict to homotopies of cells in $\mathcal{T}_{g}\setminus \mathcal{P}_{g}^{X}$. Alternatively, the dual can be constructed by gluing 1-sided duals onto any boundary components of $\mathrm{Min}(C)$, as explained below.\\

When the dual $\mathrm{Min}(C)$ to $\mathcal{P}_{g}^{X}$ at $p$ has nonempty boundary, it is also possible that $C$-regularity breaks down at $p$. In this case, the dual is not uniquely defined; it is also possible to use $\mathrm{Min}(C')$ in place of $\mathrm{Min}(C)$, where $C'\subset C$ is chosen such that $p\in \mathrm{Min}(C')$ and $\mathrm{Span}\{\nabla L(c)(p)\ |\ c\in C\}=\mathrm{Span}\{\nabla L(c)(p)\ |\ c\in C'\}$. This is illustrated schematically in Figure \ref{3possibilities}. The topological properties of duals of interest here depend only on whether the intersection with $\partial\mathcal{T}_{g}^{\epsilon_{M}}$ represents a nontrivial homology class in $\mathcal{C}_{g}^{\circ}$; this is independent of the choices, and does not require $\mathrm{Min}(C)$ to be a cell.  \\ 

\begin{figure}
\centering
\includegraphics[width=7cm]{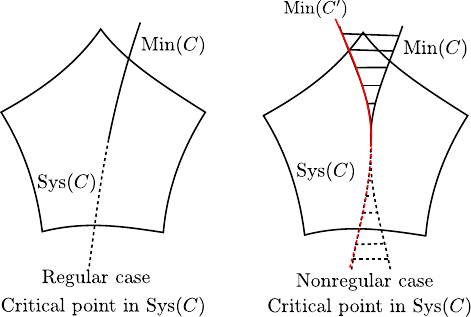}
\caption{A schematic representation of the different ways $\mathrm{Min}(C)$ can intersect $\mathrm{Sys}(C)$ at a critical point in a locally top-dimensional cell of $\mathcal{P}^{X}_{g}$. A set of minima $\mathrm{Min}(C')$ with $C'\subset C$ is shown in red.}
\label{3possibilities}
\end{figure}

\textbf{Duals and the embedding of Harvey's complex of curves.} Harvey's curve complex $\mathcal{C}_{g}$ has the homotopy type of an infinite wedge of spheres of dimension $2g-2$, \cite{Harer}. Since $\mathcal{C}_{g}$ is homotopy equivalent to $\partial\mathcal{T}_{g}^{\epsilon_{M}}$, \cite{Ivanov}, a dual cell of dimension less than $2g-1$ can always be homotoped into the thin part of $\mathcal{T}_{g}$ by a homotopy that fixes points in the intersection with the thin part of $\mathcal{T}_{g}$. As explained in \cite{JustVCD}, duals to cells of $\mathcal{P}^{X}_{g}$ are used in the construction of a $\Gamma_{g}$-equivariant deformation retraction of $\mathcal{P}^{X}_{g}$. This is done as follows:  Denote by $\mathcal{B}(p)$ a dual cell intersecting $\mathcal{P}^{X}_{g}$ in a point $p$ in a locally top-dimensional cell. If $\mathcal{B}(p)\cap \mathcal{T}_{g}^{\epsilon_{M}}$ can be homotoped into the thin part of $\mathcal{T}_{g}$ by a homotopy that keeps points in the thin part of $\mathcal{T}_{g}$ in the thin part, this homotopy moves $\mathcal{B}(p)$ off $\mathcal{P}^{X}_{g}$. This implies that $\mathcal{P}^{X}_{g}$ has nonempty boundary. This boundary is the starting point of a deformation retraction. $\Gamma_{g}$-equivariance of the deformation retraction is achieved by using the structure of the level sets of $f_{\mathrm{sys}}$ on the unstable manifolds to define the deformation retraction. Boundary cells will be assigned 1-sided duals as explained below, which are glued together to construct duals of locally top-dimensional cells as they appear in the deformation retraction.\\

\begin{figure}
\centering
\includegraphics[width=0.4\textwidth]{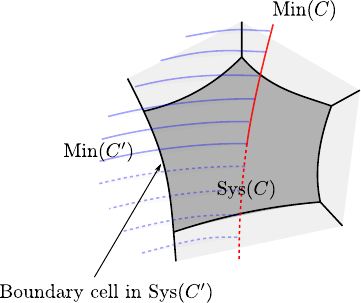}
\caption{A 1-sided dual to a boundary cell is indicated with blue stripes.}
\label{1-sided}
\end{figure}

Throughout this section, in pathological cases, the duals constructed could have quite complicated geometry, might not be submanifolds and might depend on choices. However, all that really matters are the intersection properties with the relevant subcomplexes of $\mathcal{P}_{g}^{X}$ and whether they determine trivial or nontrivial homology classes of $H_{2g-2}(\mathcal{C}_{g}; \mathbb{Z})$. These properties do not depend on choices or geometry.\\

\textbf{1-sided duals.} It might be the case that $\mathcal{P}_{g}^{X}$ has dimension larger than the lower bound given by the virtual cohomological dimension of $\Gamma_{g}$ given by $4g-5$. If so, there is an algorithm explained in \cite{JustVCD} how to obtain a series of $\Gamma_{g}$-equivariant deformation retractions of $\mathcal{P}_{g}^{X}$ onto a complex of dimension $4g-5$. Following the steps of this construction, it will now be explained how to define duals to cells that arise after performing deformation retractions on $\mathcal{P}_{g}^{X}$, using 1-sided duals. Choose and fix a $\Gamma_{g}$-orbit of boundary cells of a locally top-dimensional unstable manifold $\mathcal{M}(p)$ of a critical point $p\in \mathrm{Sys}(C)$ of dimension $d+1$. Denote by $\mathcal{B}(p)$ a dual to $\mathcal{P}_{g}^{X}$ at $p$. As discussed above, $\mathcal{B}(p)$ is either $\mathrm{Min}(C)$ or $\mathrm{Min}(C)$ with any boundary components homotoped into the thin part of $\mathcal{T}_{g}^{\epsilon_{M}}$.\\

The boundary cell is contained in the unstable manifold in $\mathcal{P}_{g}^{X}$ of a critical point $p_{1}$ with set of systoles $C_{1}$. Here $p_{1}$ has index one more than $p$. Usually $C_{1}\subsetneq C$, and so by Lemma \ref{14}, $\mathrm{Min}(C)\subset \overline{\mathrm{Min}(C_{1})}$, as illustrated schematicaly in Figure \ref{1-sided}. The set of minima $\mathrm{Min}(C_{1})$ is called the \textit{1-sided dual} of the boundary cell at $p_{1}$. \\

When $C_{1}$ is not a subset of $C$, the 1-sided dual is constructed by gluing together two or more sets of minima, as follows: Suppose $\gamma$ is the flowline from $p$ to $p_{1}$ of the flow generated by $X$ in $\mathcal{M}(p)$. If necessary, assume the $\Gamma_{g}$-orbit of the path $\gamma$ has been perturbed slightly so $\gamma$ does not pass through any strata at points where these strata have intersection with $\mathcal{M}(p)$ of codimension more than 1 in $\mathcal{M}(p)$. Then the 1-sided dual is constructed by gluing together the sets of minima labelled by the same sets of curves as the strata, in the order that they appear along $\gamma$. This is illustrated schematically in Figure \ref{relay}. To make this work, use is made of Lemma \ref{14} and the observation that the set of systoles at a point on the boundary of two strata $\mathrm{Sys}(C')$ and $\mathrm{Sys}(C'')$ contains $C'\cup C''$. The union of sets of minima are not necessarily disjoint unions, and the 1-sided duals might not be cells.\\

\begin{figure}
\centering
\includegraphics[width=0.6\textwidth]{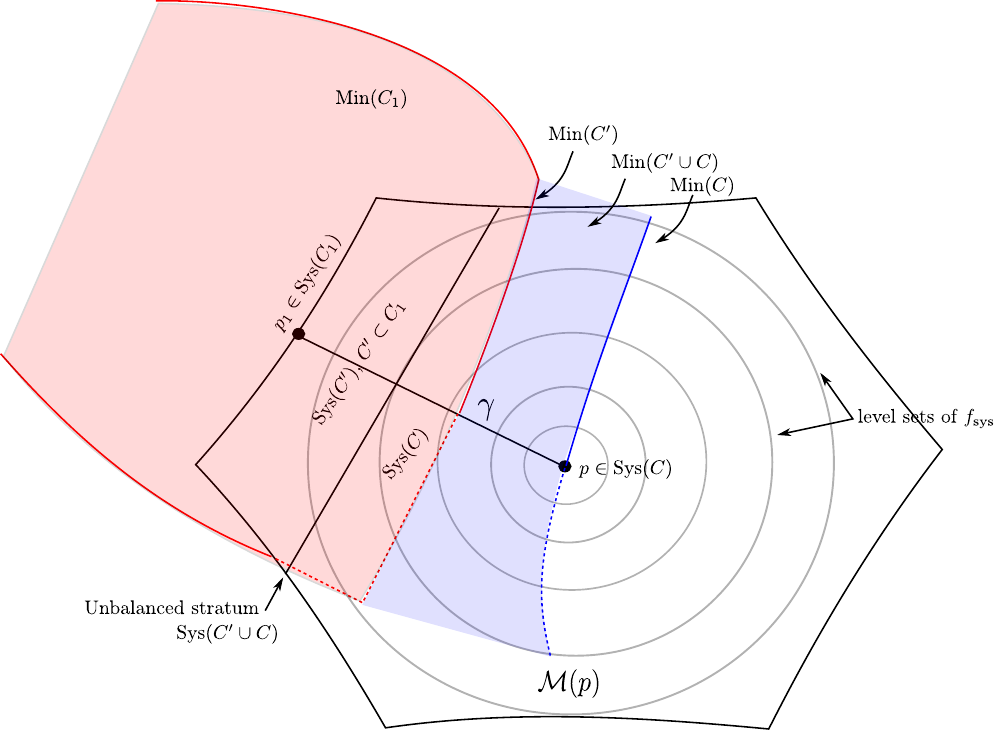}
\caption{Constructing a 1-sided dual. The path $\gamma$ from one critical point $p$ to a critical point on $\partial \mathcal{M}(p)$ with index 1 larger than the index of $p$ intersects an unbalanced stratum $\mathrm{Sys}(C'\cup C)$. The figure is intended to indicate that the shaded sets of minima intersect $\mathcal{M}(p)$ along the path $\gamma$.}
\label{relay}
\end{figure}

In order to perform the deformation retraction and to define the duals to any locally top-dimensional cells that emerge after the deformation retraction, as explained in \cite{JustVCD}, two cases are considered. The first case is that there is only one representative of the $\Gamma_{g}$-orbit of the boundary cell of the subcomplex on the boundary of the larger dimensional stable manifold, and the second case is that there is more than one such orbit. \\

\begin{figure}
\centering
\includegraphics[width=0.8\textwidth]{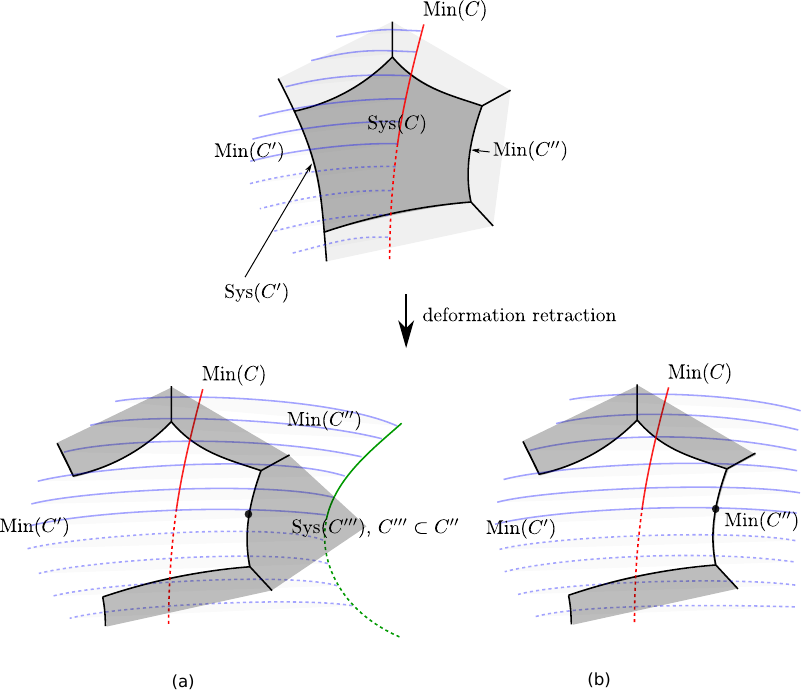}
\caption{In the case of a locally top-dimensional cell $\mathrm{Sys}(C)$ of $\mathcal{P}_{g}$ with a single boundary cell in the $\Gamma_{g}$-orbit containing $\mathrm{Sys}(C')$, part (a) illustrates the construction of a 1-sided dual intersecting the boundary in a critical point inside a boundary cell that appeared after the deformation retraction, and part (b) the construction of a dual to a locally top-dimensional cell that appeared after the deformation retraction. }
\label{simpledeformation}
\end{figure}

\textbf{Case 1. Only one representative of the $\Gamma_{g}$-orbit on the boundary of $\mathcal{M}(p)$.} In this case, the deformation retraction is obtained by deleting the $\Gamma_{g}$-orbits of $\mathcal{M}(p)$ and the boundary cell. For any boundary cells or locally top-dimensional cells that appear after this deformation retraction, the construction of duals and 1-sided duals is illustrated schematically in Figure \ref{simpledeformation}. In Figure \ref{simpledeformation} (a) it was assumed that $p\in \mathrm{Sys}(C)$, $p_{3}\in \mathrm{Sys}(C''')$ and the set of systoles at $p_{2}$ is a set $C''$ containing $C\cup C'''$. When this is not the case, the 1-sided dual through $p_{2}$ is constructed by attaching the 1-sided dual to through $p_{1}$ along $\mathrm{Min}(C)$ to a union of sets of minima. As in the construction above, this union of sets of minima correspond to the strata along the path $\gamma$, where $\gamma$ is a path in $\mathcal{P}_{g}^{X}$ obtained by concatenating the flowline from $p$ to $p_{2}$ followed by the reverse flowline from $p_{2}$ to $p_{3}$. \\

When the critical point $p_{2}$ is contained in a locally top-dimensional cell after the deformation retraction as illustrated in Figure \ref{simpledeformation} (b), the dual through $p_{2}$ is obtained by gluing the 1-sided duals through $p_{1}$ and $p_2$ along the dual through $p$ (recall this is $\mathrm{Min}(C)$). Note that $\mathrm{Min}(C)$ is contained in or on the boundary of each of the 1-sided duals.\\

\begin{figure}
\centering
\includegraphics[width=0.8\textwidth]{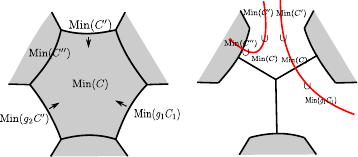}
\caption{A schematic illustration of the construction of duals and 1-sided duals, for a locally top-dimensional cell of $\mathcal{P}_{g}^{X}$ with three boundary cells in the same $\Gamma_{g}$-orbit on its boundary. The red lines indicate cross sections of duals and 1-sided duals, and the arrows on the left indicate the direction of the vector field $X$.}
\label{definingduals}
\end{figure}

\textbf{Case 2: more than one element of the $\Gamma_{g}$-orbit  on the boundary of $\mathcal{M}(p)$.} In this case, there is a finite subgroup $G(p)$ of $\Gamma_{g}$ that maps $\mathcal{M}(p)$ to itself setwise. The subgroup $G(p)$ fixes $p$, maps flaps to flaps, flowlines to flowlines, and critical points on the boundary of $\mathcal{M}(p)$ to critical points of the same index. As explained in detail in \cite{JustVCD}, this makes it possible to construct a fundamental domain $F(p)$ of the action of $G(p)$ on $\mathcal{M}(p)$. The boundary of the fundamental domain intersects the interior of $\mathcal{M}(p)$ along a set of flowlines emanating from $p$. When $G(p)$ does not have a subgroup that stabilises $p_{1}$, $F(p)$ has the boundary cell containing $p$ on its boundary. Otherwise, $\partial F(p)$ has a fundamental domain of the action of $G(p)$ on the boundary cell. This fundamental domain intersects the interior of $\mathcal{M}(p_{1})$ along a set of flowlines emanating from $p_{1}$. In the former case, the deformation retraction is obtained by deleting the interiors of the $\Gamma_{g}$-orbit of $F(p)$, as well as the $\Gamma_{g}$-orbit of the boundary cell. In the latter case, the deformation retraction is obtained by deleting the interiors of the $\Gamma_{g}$-orbit of $F(p)$, as well as the $\Gamma_{g}$-orbit of the interior of the fundamental domain of $\mathcal{M}(p_{1})$. This is followed by a further deformation retraction that collapses any remaining $\Gamma_{g}$-orbits of flowlines with an endpoint in the boundary.\\

The right hand side of Figure \ref{definingduals} indicates the construction of the duals to any locally top-dimensional cells that emerge after a deformation retraction. These locally top-dimensional cells have dimension 1  less than the dimension of $\mathcal{M}(p)$, and the duals are constructed by gluing 1-sided duals along the dual to $\mathcal{M}(p)$. The 1-sided duals through $p_{1}$ and $g_{1}p_{1}$, $g_{1}\in \Gamma_{g}$, taking $F(p)$ to an adjacent fundamental domain are glued together to form the dual of a new locally top-dimensional cell. When a new boundary cell is created by the deformation retraction, the 1-sided dual is obtained as indicated on the left of Figure \ref{definingduals}. As in earlier constructions, this might involve constructing a path $\gamma$ by concatenating two flowlines, and gluing together sets of minima defined by the sets of curves labelling the strata along $\gamma$.\\



The image $\mathcal{P}_{g,1}^{X}$ of this $\Gamma_{g}$-equivariant deformation retraction has a piecewise-smooth, $f_{\mathrm{sys}}$-increasing $\Gamma_{g}$-equivariant flow, obtained by restricting the flow on $\mathcal{P}_{g}^{X}$. It is therefore possible to iterate this construction, to obtain $\mathcal{P}_{g,2}^{X}$, $\mathcal{P}_{g,3}^{X}, \ldots$ until a $\mathcal{P}_{g,n}^{X}$ is obtained with empty boundary. \\

\begin{remark}
Theorems \ref{neededforduality} and \ref{onlyfaces} suggest that knowing the set of sets of systoles $\{C_{1}, C_{2}, \ldots\}$ at all the local maxima, $\{p_{1}, p_{2}, \ldots\}$, determines a $\Gamma_{g}$-equivariant ``pinched cell decomposition'' of $\mathcal{T}_{g}$ by the pinched cells $\{\mathrm{Min}(C_{1}), \mathrm{Min}(C_{2}), \ldots\}$. This is nearly true; as illustrated in Figure \ref{relay}, it is also necessary to know the systoles on the unbalanced strata in $\mathcal{P}_{g}$. The sets of systoles on such unbalanced strata give sets of minima that act as connector pieces between the sets of minima $\{\mathrm{Min}(C_{1}), \mathrm{Min}(C_{2}), \ldots\}$. The author is not aware of any examples of unbalanced strata in $\mathcal{P}_{g}$, although they presumably exist. An algorithm is given in \cite{SchmutzVoronoi} for calculating all the local extrema of $f_{\mathrm{sys}}$ and their systoles.
\end{remark}


\bibliography{spinebib2}
\bibliographystyle{plain}
\end{document}